\documentclass[reqno,12pt]{amsart} 
\usepackage{vmargin}
\setpapersize{A4}

\usepackage{amsmath} 

     \usepackage{amsfonts}   

\usepackage{amssymb}      

\usepackage{eufrak}      

\usepackage{amscd}      

\usepackage{amsthm}      
   
\usepackage{epsfig}      

\usepackage{amstext}      

\usepackage[all,line,dvips]{xy} 
\CompileMatrices 

\newcommand{\id}{\operatorname{id}}

\newcommand{\val}{\operatorname{val}}
\newcommand{\VAL}{\operatorname{VAL}}

   \theoremstyle{plain}
   \newtheorem{thm}{Theorem}[section]
   \newtheorem{prop}[thm]{Proposition}
   \newtheorem{lemma}[thm]{Lemma}  
   
   \theoremstyle{definition}

   \newtheorem{example}[thm]{Example}
   \theoremstyle{remark}
   
   \newtheorem{remark}[thm]{Remark}

\newtheorem{assumption}[thm]{Assumption}

\usepackage{graphicx}

   \numberwithin{equation}{section}

        \date{\today}

\title[KMS states and conformal measures]{KMS states and conformal measures}
  
\author{Klaus Thomsen}

\date{\today}

\email{matkt@imf.au.dk}
\address{Institut for matematiske fag, Ny Munkegade, 8000 Aarhus C, Denmark}

\begin{document}

\maketitle

\begin{abstract} From a non-constant holomorphic map on a connected
  Riemann surface we construct an \'etale second countable locally
  compact Hausdorff 
  groupoid whose associated groupoid $C^*$-algebra admits a
  one-parameter group of automorphisms with the property that its
  KMS states corresponds to conformal measures in the sense of
  Sullivan. In this way certain quadratic polynomials give rise to quantum statistical models
  with a phase transition arising from spontaneous symmetry breaking.  
\end{abstract}

\section{Introduction}\label{sec0} It was shown by D. Sullivan,
\cite{S}, that for any rational map $R$ on the Riemann sphere there is a
Borel probability measure $m$ on the Julia set $J_R$ and an exponent
$\delta \in ]0,2]$ such that
$$
m(R(A)) = \int_A \left|R'(z)\right|^{\delta} \ dm(z)
$$
for every Borel subset $A \subseteq J_R$ where $R$ is injective. Such
measures were subsequently called \emph{conformal} and many results
have been obtained about them, in particular results on their
uniqueness or non-uniqueness and on the values of the exponent
$\delta$ for various classes of rational maps. Nonetheless it seems
fair to say that they remain rather mysterious in general. The main purpose
with this paper is to relate these measures to KMS states of a
one-parameter group of automorphisms on a $C^*$-algebra naturally
associated to the rational map. Both in terms of intention and tools the approach we take is much
in the spirit of D. Ruelle who did a similar thing for
Gibbs measures of hyperbolic diffeomorphisms in \cite{Ru}. Thus we
associate to a rational map, and in fact to any non-constant holomorphic map on a
connected Riemann surface, an \'etale groupoid such that its (reduced)
$C^*$-algebra can be defined as described by Renault in \cite{Re}. The
same construction works for any totally invariant subset which is
locally compact in the relative topology and has no isolated
points. For rational maps on the Riemann sphere this means that as far
as the construction of the groupoid and its $C^*$-algebra is concerned, we can
treat the Julia set or the Fatou set in the same way as the whole sphere. By
construction the groupoid comes equipped with a natural real-valued
homomorphism and the corresponding one-parameter group of
automorphisms has the property that there is a one-to-one correspondence between
the non-atomic conformal measures with exponent $\delta$ for the
holomorphic map and a face
in the weak$^*$ closed set of its $\delta$-KMS states. The
correspondence extends to atomic measures, but for them the relation
is more complicated, and in particular the map from KMS states to
measures is generally not injective. As an illustration of the general results we show how certain
quadratic polynomials in this way give rise to quantum statistical models with a
phase transition arising from spontaneous symmetry breaking in the
sense of Bost and Connes, \cite{BoC}.

The main tool for the identification of the KMS states is a recent
result of Neshveyev, \cite{N}, in which he extends results of Renault
to obtain a general description of the KMS states for the
one-parameter group of automorphisms arising from a real-valued
homomorphism on a second countable locally compact \'etale
groupoid. Since the groupoids we construct have the property that all
isotropy groups are abelian and the points in the unit space with
non-trivial isotropy group are at most countable, the results of
Neshveyev can be transferred from the full to the reduced
groupoid $C^*$-algebra and be given slightly more
detailed formulations. We do this in the first section of the paper
before we move to the main part where we construct the \'etale
groupoid of a holomorphic map and the relevant one-parameter group of
automorphisms, which we call \emph{the conformal action}, on its
$C^*$-algebra. We obtain a general description of the $\beta$-KMS
states for the conformal action
when $\beta \neq 0$ and illustrate it by considering the restriction to
the Julia set $J_R$ of a quadratic polynomial $R$ which satisfies the Collet-Eckmann
condition. Thanks to results of Graczyk and Smirnov, \cite{GS1},
\cite{GS2}, we can in this case give a 
complete description of the $\beta$-KMS states for the conformal
action for positive $\beta$. When the critical point is pre-periodic
there is only one KMS state, corresponding to the Sullivan
measure with exponent equal to the Hausdorff dimension $HD(J_R)$ of $J_R$. When
the critical point is not pre-periodic there is no $\beta$-KMS state
for $0 < \beta < HD(J_R)$, a unique one for $\beta = HD(J_R)$,
corresponding again to the Sullivan measure, and then two extremal $\beta$-KMS
states for any $\beta > HD(J_R)$. The presence of the latter KMS states
is caused by the summability of the Poincar\'e series for the critical
point which was established in \cite{GS2}.

The notion of conformality for measures related to dynamical systems
has been generalised in various ways, and we show in the last section
that some of these
generalisations can also be covered by the approach taken here. In
particular, when the map in question is a rational map on the Riemann
sphere we obtain a complete description of the
KMS states for the gauge action which comes naturally from the
construction of the $C^*$-algebra. This allows a direct comparison
between our construction and that of Kajiwara and Watatani in
\cite{KW} where they construct $C^*$-algebras from rational maps on
the Riemann sphere via Hilbert modules and the Cuntz-Pimsner
construction. The KMS states of the gauge action on their algebras
were described in \cite{IKW}, and there are both similarities and
differences in the structure of the KMS states when compared to the
findings in this paper. The differences show that there is generally no natural
(gauge-preserving) way to pass from their $C^*$-algebras to the ones
constructed here.      

\smallskip

\emph{Acknowledgement.} I am grateful to an anonymous referee who
found a mistake in the first version of Section \ref{gauge}.

\section{Groupoid $C^*$-algebras and KMS states}\label{sec3}

\subsection{Groupoid $C^*$-algebras and one-parameter automorphism groups}\label{sec3a}

Let $G$ be an \'etale second countable locally compact Hausdorff groupoid with unit
space $G^{(0)}$. Let $r : G \to G^{(0)}$ and $s : G \to G^{(0)}$ be
the range and source maps, respectively. For $x \in G^{(0)}$ put $G^x = r^{-1}(x), \ G_x = s^{-1}(x) \ \text{and} \ G^x_x =
s^{-1}(x)\cap r^{-1}(x)$. Note that $G^x_x$ is a group, the \emph{isotropy group} at $x$. The space $C_c(G)$ of continuous compactly supported
functions is a $*$-algebra when the product is defined by
$$
(f_1 * f_2)(g) = \sum_{h \in G^{r(g)}} f_1(h)f_2(h^{-1}g)
$$
and the involution by $f^*(g) = \overline{f\left(g^{-1}\right)}$. The
\emph{full groupoid $C^*$-algebra} $C^*(G)$ is the completion of $C_c(G)$ with
respect to the norm
$$
\left\|f\right\| = \sup_{\pi} \left\|\pi(f)\right\|,
$$
where the supremum is taken over all $*$-representations $\pi$ of
$C_c(G)$ on a Hilbert space. 
To define the \emph{reduced groupoid
  $C^*$-algebra} $G^*_r(G)$ let $x\in G^{(0)}$. There is a
representation $\pi_x$ of $C_c(G)$ on the Hilbert space $l^2(G_x)$ of
square-summable functions on $G_x$ given by 
\begin{equation}\label{pix}
\pi_x(f)\psi(g) = \sum_{h \in G^{r(g)}} f(h)\psi(h^{-1}g) .
\end{equation}
$C^*_r(G)$ is the completion of $C_c(G)$ with respect to the norm
$$
\left\|f\right\|_r = \sup_{x \in G^{(0)}} \left\|\pi_x(f)\right\| .
$$  
Since $\|f\|_r \leq \|f\|$ there is a canonical surjection $\lambda : C^*(G)
\to C^*_r(G)$. 

We are mainly interested in the reduced groupoid
$C^*$-algebra and have primarily introduced the full version in order
to use the results of Neshveyev, \cite{N}. The inclusion $C_c\left(G^{(0)}\right) \subseteq C_c(G)$
extends to embeddings of $C_0\left(G^{(0)}\right)$ into both $C^*(G)$
and $C^*_r(G)$ and $\lambda$ restricts to an isomorphism between
the two copies of $C_0\left(G^{(0)}\right)$ which we therefore
identify. 
Note that the map $C_c(G) \to C_c\left(G^{(0)}\right)$ which
restricts functions to $G^{(0)}$ extends to a conditional expectation $P
: C^*_r(G) \to C_0\left(G^{(0)}\right)$. 

Let $c : G \to \mathbb R$ be a continuous homomorphism, i.e. $c$ is
continuous and $c(gh) = c(g) + c(h)$ when $s(g) = r(h)$. For each $t
\in \mathbb R$ we can then define an automorphism $\sigma_t^c$ of
$C_c(G)$ such that
$$
\sigma^c_t(f)(g) = e^{itc(g)}f(g) .
$$ 
For each $x \in G^{(0)}$ the same expression defines a unitary $u_t$ in
$l^2(G_x)$ such that $u_t\pi_x(f)u_t^* =
\pi_x\left(\sigma^c_t(f)\right)$ and it follows therefore that
$\sigma^c_t$ extends by continuity to an automorphism $\sigma^c_t$ of
$C^*_r\left(G\right)$. It is easy to see that $\sigma^c =
\left(\sigma^c_t\right)_{t \in \mathbb R}$ is a continuous
one-parameter group of automorphisms of
$C^*_r\left(G\right)$. 
In the same way $c$ gives also rise to a continuous
one-parameter group of automorphisms of the full groupoid $C^*$-algebra
$C^*(G)$ 
and $\lambda : C^*(G) \to C^*_r(G)$ is then equivariant.

\subsection{KMS states}\label{A}

Let $A$ be a $C^*$-algebra and $\alpha_t, t \in \mathbb R$, a continuous
one-parameter group of automorphisms of $A$. Let
$\beta \in \mathbb R$. A state $\omega$ of $A$ is a \emph{$\beta$-KMS state} when 
$$
\omega(a\alpha_{i\beta}(b)) = \omega(ba)
$$
for all elements $a,b$ of a dense $\alpha$-invariant $*$-algebra of $\alpha$-analytic elements,
cf. \cite{BR}. In this section we use the results of Neshveyev from
\cite{N} to study the KMS-states of $\sigma^c$ on $C^*_r(G)$ under a
couple of
additional assumptions which will hold for the groupoid constructed in
the following section. Since it simplifies several things we will
only consider the case $\beta \neq 0$.

For each $x \in G^{(0)}$ the subset $G^x_x$ is a closed discrete
subgroup of $G$. When $G_x^x$ is amenable the full
and reduced group $C^*$-algebras of $G^x_x$ coincide, i.e. the canonical
homomorphism $\lambda :
C^*\left(G^x_x\right) \to C^*_r\left(G_x^x\right)$ is an
isomorphism. This leads to the following

\begin{lemma}\label{lambda} Assume that $G^x_x$ is amenable for all $x
  \in G^{(0)}$. It follows that every state $\omega$ of $C^*(G)$ with
  $C_0\left(G^{(0)}\right)$ in it is centraliser (i.e. $\omega(af) =
  \omega(fa)$ for all $f \in C_0\left(G^{(0)}\right), a \in C^*(G)$)
  factorises through $C^*_r(G)$.
\end{lemma}
\begin{proof} It suffices to show that $\left|\omega(f)\right| \leq \|f\|_r$
  for all $f \in C_c(G)$. It follows from Theorem 1.1 of \cite{N} that
  there is a
  Borel probability measure $m$ on $G^{(0)}$ and for each $x \in
  G^{(0)}$ a state $\omega_x$ on $C^*(G^x_x)$ such that $G^{(0)} \ni x
\mapsto \omega_x\left(f|_{G^x_x}\right)$ is Borel and
\begin{equation}\label{kul}
\omega(f) = \int_{G^{(0)}} \omega_x\left(f|_{G^x_x}\right) \ dm(x)
\end{equation}
for all $f \in C_c(G)$. Note that 
$\left|\omega_x\left(f|_{G^x_x}\right)\right| \leq
\left\|f|_{G^x_x}\right\|_{C^*\left(G^x_x\right)} =
\left\|f|_{G^x_x}\right\|_{C^*_r\left(G^x_x\right)}$ since $G_x^x$ is
amenable. It follows from the definition of $\pi_x$,
cf. (\ref{pix}), that $\left\|f|_{G^x_x}\right\|_{C^*_r\left(G^x_x\right)} =
\left\|P_x\pi_x(f)P_x\right\|$ where $P_x : l^2(G_x) \to l^2(G^x_x)$
is the orthogonal projection. Therefore
$\left|\omega_x\left(f|_{G^x_x}\right)\right| \leq
\left\|f|_{G^x_x}\right\|_{C^*_r\left(G^x_x\right)} \leq \left\|\pi_x(f)\right\|$
for each $x$ and then (\ref{kul}) shows that $\left|\omega(f)\right| \leq \|f\|_r$ as desired. 
\end{proof}

Note that
the measure $m$ in \ref{kul} is determined by the condition that
$\omega(f) = \int_{G^{(0)}} f \ dm$ for all
$C_c\left(G^{(0)}\right)$. We say that $m$ is the measure
\emph{associated with $\omega$}. 

Since every KMS state for $\sigma^c$ has $C_0\left(G^{(0)}\right)$ in
its centraliser it follows from Lemma \ref{lambda} that the map $\lambda : C^*(G) \to C^*_r(G)$
induces a bijection from the KMS states of the one-parameter group
$\sigma^c$ on $C^*_r(G)$ onto its KMS states on $C^*(G)$. The following definition characterises the
Borel probability measures which arise from KMS states in this way. Let $W
\subseteq G$ be an open bi-section, i.e. an open subset such that $r :
W \to G^{(0)}$ and $s: W \to G^{(0)}$ are both injective. Then $r : W \to r(W)$ is a homeomorphism and we
denote its inverse by $r_W^{-1}$. Let $\beta
\in \mathbb R \backslash \{0\}$. We say that a finite Borel measure $m$ on $G^{(0)}$ is
\emph{$(G,c)$-conformal with exponent $\beta$} when
\begin{equation}\label{Gconf}
m\left(s(W)\right) = \int_{r(W)} e^{\beta c\left(r_W^{-1}(x)\right)} \ dm(x)
\end{equation}
for every open bi-section $W$ of $G$. Note that in the terminology
used in \cite{N} this condition means that $m$ is
quasi-invariant with Radon-Nikodym cocycle $e^{-\beta c}$.

Assume now that 
\begin{enumerate} 
\item[a)] $G^x_x$ is abelian for all $x \in G^{(0)}$ and
\item[b)] $G_x^x = \{x\}$ for all but at most countably many $x \in
  G^{(0)}$.
\end{enumerate}

In the following we call a finite Borel measure $m$ on $G^{(0)}$
\emph{non-atomic} when $m\left(\{x\}\right) = 0$ for all $x \in
G^{(0)}$ and \emph{purely atomic} when there is a Borel set $A
\subseteq G^{(0)}$ such that $m(A) = m\left(G^{(0)}\right)$ and $m(\{a\}) > 0$ for all $a
\in A$. Similarly we
will say that a KMS state for $\sigma^c$ is \emph{non-atomic} when its
associated measure is non-atomic, and \emph{purely atomic} when it is
purely atomic. 

. 

\begin{lemma}\label{decomp!!} Let $m$ be a finite Borel measure on
  $G^{(0)}$ and let $m = m^c + m^a$ be a decomposition of $m$ into the
  sum of the
  non-atomic measure $m^c$ and the purely atomic measure $m^a$. It
  follows that $m$ is $(G,c)$-conformal with exponent $\beta$ if and only
  if $m^c$ and $m^a$ both are.
\end{lemma}
\begin{proof} Assume that $m$ is $(G,c)$-conformal with exponent
  $\beta$. Let $W \subseteq G $ be an open bi-section. When $V$ is an
  open subset of $r(W)$, the set $r_W^{-1}(V)$ is an open bi-section so it
  follows that $m\left(s(r_W^{-1}(V))\right) = \int_{V} e^{\beta c\left(r_W^{-1}(x)\right)} \
  dm(x)$ for every open subset $V \subseteq r(W)$. By (outer) regularity of the measures on $r(W)$ given by $B \mapsto
  m\left(s\left(r_W^{-1}(B)\right)\right)$ and $B \mapsto \int_{B}
  e^{\beta c\left(r_W^{-1}(x)\right)} \ dm(x)$ it follows that
\begin{equation}\label{borelconf}  
m\left(s(r_W^{-1}(B))\right) = \int_{B} e^{\beta c\left(r_W^{-1}(x)\right)} \
  dm(x)
\end{equation}
for every Borel subset $B \subseteq r(W)$. Let $E$ be the set of atoms
for $m$. Since $G$ is covered by bi-sections it follows from
(\ref{borelconf}) that
  $s\left(r^{-1}(E)\right) = E = r\left(s^{-1}(E)\right)$; a
  conclusion which can be put back into (\ref{borelconf}) to give
\begin{equation*}
\begin{split}
&m^c(s(W)) = m\left(s(W) \backslash E\right) = m\left(s(r_W^{-1}\left(r(W)
    \backslash E\right)\right) \\
&= \int_{r(W) \backslash E} e^{\beta
  c\left(r_W^{-1}(x)\right)} \ dm(x) = \int_{r(W)} e^{\beta c\left(r_W^{-1}(x)\right)} \ dm^c(x) .
\end{split}
\end{equation*}
Similarly, $m^a(s(W)) = \int_{r(W)} e^{\beta c\left(r_W^{-1}(x)\right)} \ dm^a(x)$
and we conclude that both $m^c$ and $m^a$ are $(G,c)$-conformal with
exponent $\beta$. The converse is trivial. 
\end{proof}

Let $x \in G^{(0)}$. The \emph{$G$-orbit} $Gx$ of $x$ is the set
$G x=r(G_x)$ . We say that $Gx$ is \emph{consistent} when
$c\left(G^x_x\right) = 0$. When this holds we can define a map
$l_x : Gx \to ]0,\infty[$ such that $l_x(z) = e^{-c(g)}$ where
$g$ is any element of $r^{-1}(z) \cap G_x$. We say that $Gx$ is
\emph{$\beta$-summable} when it is consistent and
$$
\sum_{z \in Gx}  l_x(z)^{\beta} < \infty.
$$ 

\begin{lemma}\label{sumconsist} Let $m$ be a finite Borel measure on
  $G^{(0)}$ which is $(G,c)$-conformal with exponent $\beta$. Assume that
  $m\left(\{x\}\right) > 0$. It follows that the $G$-orbit $Gx$ is consistent and
  $\beta$-summable. Furthermore,
  $m(\{y\}) > 0$ for all $y \in Gx$.
\end{lemma}
\begin{proof} Let $g \in G^x_x$. There is an open bi-section $W
  \subseteq G$ such that $g \in W$. It follows therefore from
  (\ref{borelconf}) that $m(\{x\}) = e^{\beta c(g)} m(\{x\})$ which
  implies that $c(g) = 0$. Hence $Gx$ is consistent. Similarly,
  we find that $m(\{x\}) = e^{\beta c(\xi)}m(\{y\}) =
  l_x(y)^{-\beta}m(\{y\})$ when $y = r(\xi x)$,
  proving that $m(\{y\}) > 0$ for all $y \in Gx$. Finally, observe
  that $m(\{x\})\sum_{z \in Gx} l_x(z)^{\beta} = m\left(Gx\right)$
  which implies that $\sum_{z\in Gx} l_x(z)^{\beta} =
  \frac{m\left(Gx\right)}{m(\{x\})} < \infty$.
\end{proof}

Consider then a consistent and $\beta$-summable $G$-orbit
$\mathcal O = Gx$. When we denote the Dirac measure at $z$ by $\delta_z$ we
can define a Borel probability measure $m_{\mathcal O}$ on
$G^{(0)}$ such that
\begin{equation}\label{mathcalo}
m_{\mathcal O} = \left( \sum_{z \in Gx} l_x(z)^{\beta} \right)^{-1}
\sum_{z \in Gx} l_x(z)^{\beta} \delta_{z} .
\end{equation}

It is straightforward to check that $m_{\mathcal O}$ is $(G,c)$-conformal
with exponent $\beta$. Let $\varphi$ be a state on
$C^*\left(G_x^x\right)$. For each $z \in \mathcal
O$ choose an element $\xi_z \in G$ such that $r(\xi_z) = z$ and
$s\left(\xi_z\right) = x$. Then $\xi_zG_x^x\xi_z^{-1} = G^z_z$. Define a state $\varphi_z$ on
$C^*\left(G^z_z\right)$ such that
$$
\varphi_z(f) = \varphi\left(f^{\xi_z}\right),
$$ 
where $f^{\xi_z} \in C_c\left(G_x^x\right) \subseteq C^*(G^x_x)$ is the
function $f^{\xi_z}(y) = f\left(\xi_zy\xi_z^{-1}\right)$.  By using that
$C^*\left(G_x^x\right)$ is abelian for all $x$ by assumption a), a direct
calculation as in the proof of Theorem 1.3 in \cite{N} shows that there is a
$\beta$-KMS state $\omega_{\mathcal O}^{\varphi}$ for $\sigma^c$ such that
$$
\omega_{\mathcal O}^{\varphi}(h) = \left(\sum_{z \in \mathcal O}
  l_x(z)^{\beta}\right)^{-1}  \sum_{z \in \mathcal O} 
\varphi_z\left(h|_{G^z_z}\right)l_x(z)^{\beta}
$$ 
for all $h \in C_c(G)$. 


\begin{thm}\label{decomp!} Assume that a) and b) hold. Let $m$ be a non-atomic Borel probability
  measure on $G^{(0)}$. Assume that $m$ is $(G,c)$-conformal with exponent
  $\beta$, and let $\mathcal O_i, i \in I$, be a
  finite or countably infinite collection of consistent and $\beta$-summable
  $G$-orbits. Choose for each $i \in I$ an element $x_i \in
  \mathcal O_i$ and let $\varphi_i$ be a state on
  $C^*\left(G^{x_i}_{x_i}\right)$. Finally,
let $t_0,t_i \in ]0,1], i \in I$, be numbers such that $t_0 + \sum_{i
  \in I} t_i = 1$.

There is then a $\beta$-KMS state $\omega$ for $\sigma^c$ on $C^*_r(G)$
such that
\begin{equation}\label{decmpp}
\omega(a) = t_0 \int_{G^{(0)}} P(a) \ dm + \sum_{i \in I} t_i \omega_{\mathcal
  O_i}^{\varphi_i}(a)
\end{equation}
for all $a \in C^*_r\left(G\right)$. Conversely, any $\beta$-KMS state
$\omega$ for
$\sigma^c$ on $C^*_r(G)$ admits a unique decomposition of the form (\ref{decmpp}).
\end{thm}  
\begin{proof} It follows from Theorem 1.3 of \cite{N} that $a \mapsto
  \int_{G^{(0)}} P(a) \ dm$ is a $\beta$-KMS state. Since the same is true for each $\omega_{\mathcal
  O_i}^{\varphi_i}$ the fact that the $\beta$-KMS states constitute a
weak*-closed convex set implies that $\omega$ is a $\beta$-KMS state. 

Conversely, let $\omega$ be a $\beta$-KMS state. It follows from
Theorem 1.3 of \cite{N} that there is a Borel probability measure $m$
on $G^{(0)}$ and a Borel measurable field of states $\varphi_x$ on $C^*(G^x_x),
\ x \in G^{(0)}$,
such that 
\begin{equation}\label{ii)}
\varphi_{x}(f) = \varphi_{r(\xi)}\left(f^{\xi^{-1}}\right), \ f
\in C_c\left(G^x_x\right),
\end{equation}
for all $\xi \in G_x$ and $m$-almost every $x$, and
\begin{equation}\label{iii)}
\omega(f) = \int_{G^{(0)}} \varphi_x\left(f|_{G^x_x}\right) \ dm(x) 
\end{equation}
for all $f \in C_c(G)$. When $m$ is non-atomic it follows from
Corollary 1.2 of \cite{N} that $\omega(a) = \int_{G^{(0)}} P(a) \ dm$
for all $a$, and we are done. Assume that $m$ is purely atomic. It
follows then from Lemma
\ref{sumconsist} that there is a countable collection $\mathcal O_i, i \in I$, of
consistent and $\beta$-summable $G$-orbits such that $m = \sum_{i
  \in I} m\left( \mathcal O_i\right) m_{\mathcal O_i}$. Choose an
element $x_i \in \mathcal O_i$. It follows from (\ref{ii)}) and
(\ref{iii)})  that $\omega = \sum_{i \in I}
m\left(\mathcal O_i\right) \omega^{\varphi_{x_i}}_{\mathcal
  O_i}$. Finally, when $m$ is neither non-atomic nor purely atomic we
apply Lemma \ref{decomp!!} to get an $s
\in ]0,1[$ and $(G,c)$-conformal Borel probability measures $m^c$ and $m^a$ with
exponent $\beta$ such that $m^c$ is non-atomic, $m^a$ is purely atomic
and $m = sm^c + (1-s)m^a$. Since $m^c$ is non-atomic it follows from
assumption b) that
\begin{equation}\label{c-a}
\omega(f) =  s\int_{G^{(0)}} f(x) \ dm^c(x)  +(1-s)\int_{G^{(0)}}
\varphi_x\left(f|_{G^x_x}\right) \ dm^a(x)  
\end{equation}
for all $f \in C_c(G)$. By repeating the previous argument with $m$
replaced by $m^a$ we get a countable collection $\mathcal O_i, i \in I$, of
consistent and $\beta$-summable $G$-orbits and elements $x_i \in
\mathcal O_i$ such that
$$
\int_{G^{(0)}}
\varphi_x\left(f|_{G^x_x}\right) \ dm^a(x) = \sum_{i \in I}
m^a\left(\mathcal O_i\right) \omega^{\varphi_{x_i}}_{\mathcal
  O_i}(f)
$$ for all $f \in C_c\left(G\right)$. It follows then that
$$
\omega(a) = s \int_{G^{(0)}} P(a) \ dm^c + (1-s)  \sum_{i \in I}
m^a\left(\mathcal O_i\right) \omega^{\varphi_{x_i}}_{\mathcal O_i}(a) 
$$
for all $a \in C^*_r(G)$. This shows that $\omega$ can be decomposed as in (\ref{decmpp}).

To
 see why the decomposition (\ref{decmpp}) is unique, let $\{x_k\}$ be
 a numbering of the elements of $\mathcal O_i$. For each $l \in
 \mathbb N$ let $\{f^l_k\}$ be a real and
bounded
sequence in $C_c\left(G^{(0)}\right)$ converging point wise to the
characteristic function of $\left\{x_1,x_2, x_3,\dots, x_l\right\}$. Let $\{g_n\}$ be an
approximate unit for $C^*_r(G)$ contained in $
C_c\left(G^{(0)}\right)$ and let $a \in C^*_r(G)$. Since the
Cauchy-Schwarz inequality implies that
$$
\omega^{\varphi_j}_{\mathcal O_j}(f^l_kg_na) \leq
\omega^{\varphi_j}_{\mathcal O_j}(a^*a)\int_{G^{(0)}} {f^l_k}^2g_n^2 \ d
m_{\mathcal O_j} , 
$$
we see that $\lim_{k \to \infty} \omega^{\varphi_j}_{\mathcal
  O_j}(f^l_kg_na) = 0$ for all $l$ when $j \neq i$. It follows in the
same way that $\lim_{k \to \infty} \int_{G^{(0)}} P(f^l_kg_na)
\ dm = 0$ for all $l$, and the estimate
$$
\left| \omega^{\varphi_i}_{\mathcal O_i}(g_na)
  -\omega^{\varphi_i}_{\mathcal O_i}(f^l_kg_na)\right| \leq
\omega^{\varphi_i}_{\mathcal O_i}(a^*a)\int_{G^{(0)}} {\left(1- f^l_k\right)}^2g_n^2 \ d
m_{\mathcal O_i}
$$
shows that $\lim_{l \to \infty} \lim_{k \to \infty} \omega^{\varphi_i}_{\mathcal
  O_i}(f^l_kg_na) =  \omega^{\varphi_i}_{\mathcal
  O_i}(g_na)$. Thus
$$
\lim_{n \to \infty} \lim_{l \to \infty} \lim_{k
  \to \infty} \omega\left(f^l_kg_na\right) = \lim_{n \to \infty} t_i \omega^{\varphi_i}_{\mathcal
  O_i}(g_na) =t_i \omega^{\varphi_i}_{\mathcal
  O_i}(a).
$$ 
This shows that $t_i \omega^{\varphi_i}_{\mathcal
  O_i}$ is determined by $\omega$, and hence so is $t_i$ and $ \omega^{\varphi_i}_{\mathcal
  O_i}$. The same is then automatically true for $t_0$ and $m$.
\end{proof}

The uniqueness part of the statement makes it easy to identify the
extremal $\beta$-KMS states. In particular it follows that they are either non-atomic or purely atomic.

\section{An  amended transformation groupoid}\label{sec1}

Let $X$ be a locally compact Hausdorff space and $\psi : X \to X$ a map. Let $\mathcal
P$ be a pseudo-group on $X$. More specifically, $\mathcal P$ is a
collection of local homeomorphisms $\eta : U \to V$ between open
subsets of $X$ such that
\begin{enumerate}
\item[i)] for every open subset $U$ of $X$ the identity map $\id : U
  \to U$ is in $\mathcal P$,
\item[ii)] when $\eta : U \to V$ is in $\mathcal P$ then so is
  $\eta^{-1} : V \to U$, and
\item[iii)] when $\eta : U \to V$ and $\eta_1 : U_1 \to V_1$ are
  elements in $\mathcal P$ then so is $\eta_1 \circ \eta : U \cap
  \eta^{-1}(V\cap U_1) \to \eta_1(V \cap U_1)$.  
\end{enumerate}  
For each $k \in \mathbb Z$ we denote by
$\mathcal T_k(\psi)$ the elements $\eta : U \to V$ of $\mathcal P$ with the
property that there are natural numbers $n,m$ such that $k = n-m$ and
\begin{equation}\label{crux0} 
\psi^n(z) = \psi^m(\eta(z))  \ \ \forall z \in U.
\end{equation}
The elements of $\mathcal T =\bigcup_{k \in \mathbb Z} \mathcal T_k(\psi)$ will
be called \emph{local transfers} for $\psi$. 
We denote by $[\eta]_x$ the germ at a point $x \in X$ of an element $\eta \in \mathcal
T_k(\psi)$. Set
$$
\mathcal G_{\psi} = \left\{ (x,k,\eta,y) \in X \times \mathbb Z \times
  \mathcal P  \times X : \ \eta \in \mathcal T_k(\psi) , \ \eta(x)
  = y \right\} .
$$
We define an equivalence relation $\sim$ in $\mathcal G_{\psi}$ such
that $(x,k,\eta,y) \sim (x',k',\eta',y')$ when
\begin{enumerate}
\item[i)] $ x = x', \ y = y', \ k = k'$ and
\item[ii)] $[\eta]_x = [\eta']_x$.
\end{enumerate}
Let
$\left[x,k,\eta,y\right]$ denote the equivalence class represented
by $(x,k,\eta,y)  \in \mathcal G_{\psi}$. The quotient space 
$G_{\psi} = \mathcal G_{\psi}/{\sim}$ is a groupoid such that two elements
$\left[x,k,\eta,y\right]$ and $\left[x',k',\eta',y'\right]$ are
composable when $y= x'$ and their product is
$$
\left[x,k,\eta,y\right]\left[y,k',\eta',y'\right] =
\left[x,k+k',\eta'\circ \eta ,y'\right] .
$$ 
The inversion in $G_{\psi}$ is defined such that
$\left[x,k,\eta,y\right]^{-1} = \left[ y,-k,\eta^{-1},x \right]$. The unit space of $G_{\psi}$ can be identified with $X$ via the map $x \mapsto [x,0,\id,x]$, where $\id$ is the
identity map on $X$. When $\eta \in \mathcal T_k(\psi)$ and $U$ is an open subset of the domain
of $\eta$ we set
\begin{equation}\label{baseset}
U(\eta) = \left\{ \left[z,k,\eta,\eta(z)\right]  : \ z \in U \right\}.
\end{equation}
It is straightforward to verify that by varying $k$, $\eta$ and $U$ the sets (\ref{baseset}) constitute a base
for a topology on $G_{\psi}$. In general this topology is not
Haussdorff and to amend this we now make the following additional assumption.
\begin{assumption}\label{crux2}
When $x \in X$ and $\eta(x) = \xi(x)$ for some $\eta,
\xi \in \mathcal T_k(\psi)$, then the implication
\begin{equation*}\label{crux1}
\text{$x$
  is not isolated in} \ \left\{ y \in X : \ \eta(y) = \xi(y)\right\} \ \
\Rightarrow \ \ [\eta]_x =[\xi]_x
\end{equation*}
holds.
\end{assumption} 
Then $G_{\psi}$ is Hausdorff: Let
  $\left[x,k,\eta,y\right]$ and $\left[x',k',\eta',y'\right]$ be
  different elements of $G_{\psi}$. There are then open neighbourhood's $U$ of $x$ and $U'$ of $x'$ such that 
$U(\eta) =\left\{
    \left[z,k,\eta,\eta(z)\right] : \ z \in U\right\}$ and $U'(\eta') = \left\{
    \left[z,k',\eta',\eta'(z)\right] : \ z \in U'\right\}$ are
  disjoint. This is trivial when $(x,k,y) \neq (x',k',y')$ while it is
  a straightforward consequence of Assumption \ref{crux2} when
  $(x,k,y) = (x',k',y')$. 

Since the range and source maps are homeomorphisms from $U(\eta)$ onto
$U$ and $\eta(U)$, respectively, it follows that $G_{\psi}$ is a
locally compact Hausdorff space because $X$ is. It is also straightforward to show
that the groupoid operations are continuous so that we can
conclude the following. 

\begin{thm}\label{etale!} Let Assumption \ref{crux2} be satisfied. Then
  $G_{\psi}$ is an \'etale locally compact Hausdorff groupoid.  
\end{thm}

Assumption \ref{crux2} may be satisfied because of properties of the map
$\psi$, regardless of which pseudo-group $\mathcal P$ is used; it holds for example trivially when
$\psi$ is locally injective on $X$. When
$\psi$ is a local homeomorphism and the pseudo-group $\mathcal P$
consists of all local homeomorphisms on $X$, the groupoid $G_{\psi}$ is the
same as the one considered in increasing generality by Renault, Deaconu
and Anantharaman-Delaroche, \cite{Re},\cite{De},\cite{An}. It is
therefore also a generalisation of the classical crossed product
construction for homeomorphisms.

\section{\'Etale groupoids from holomorphic maps}\label{sec2}

Let $S$ be a connected Riemann surface and $H : S \to S$ a non-constant holomorphic map. Let
$X \subseteq S$ be a subset which is locally compact in the topology
inherited from $S$. Assume that no points are isolated in $X$ and that $X$
is totally $H$-invariant, i.e. that $H^{-1}(X) =X$, and let $H|_X : X
\to X$ denote the restriction of $H$ to $X$. Let
$\mathcal P$ be the pseudo-group on $X$ of local homeomorphisms $\xi : U
\to V$ between open subsets of $X$ with the property that there are
open subsets $U_1,V_1$ in $S$ and a bi-holomorphic map $\xi_1 : U_1
\to V_1$ such that $U_1 \cap X = U, \ V_1 \cap X = V$ and $\xi =
\xi_1$ on $U$. 
Then Assumption
\ref{crux2} is satisfied. This follows from the well-known fact that holomorphic
  maps defined on the same open connected subset of the complex plane
  must be identical if they agree on a set with a limit point in their
  domain. 
Therefore Theorem \ref{etale!} implies that $G_{H|_X}$ is an
  \'etale locally compact Hausdorff groupoid. To simplify notation we
  denote it by $G_X$. 

\subsection{$G_X$ is second countable}

To show that there is a countable base for the topology of $G_X$ we
need some preparations that are also going to be instrumental in
determining the isotropy groups.

Let $x \in S$. A \emph{conformal germ} at $x$ is a
holomorphic and injective map $\eta : U \to S$ where $U$ is an open neighbourhood of
$x$. As is well-known a conformal germ is invertible in the sense that
$\eta(U)$ is open and $\eta^{-1} : \eta(U) \to U$ is holomorphic. The
set of conformal germs at $x$ will be denoted by $\mathcal
H_x$. Two conformal germs $\eta, \eta' \in \mathcal H_x$ are
\emph{equivalent} when they agree in an open neighbourhood of $x$. We
occasionally identify $\mathcal H_x$ with $\mathcal H_x/{\sim}$;
hopefully it will be clear from the context when this happens. Let $\Delta = \left\{ z \in \mathbb C: \ |z| < 1
  \right\}$. For every $x \in S$ there are neighbourhoods $U$
and $V$ of $x$ and $H(x)$, respectively, and homeomorphisms $\varphi :
U \to \Delta, \psi : V \to \Delta$ such that $H(U) \subseteq V$,
$\varphi(x) = \psi(H(x)) = 0$, and
there is an $n \in \mathbb N$ such that
\begin{equation*}
\begin{xymatrix}{
U \ar[rr]^H \ar[d]_-{\varphi} & &V \ar[d]^-{\psi} \\
\Delta \ar[rr]_-{z \mapsto z^n} &  & \Delta }
\end{xymatrix}
\end{equation*}
commutes. This it is trivial when $H'(x) \neq 0$ and it follows from
B\"ottchers theorem, Theorem 9.1 in \cite{Mi}, when $H'(x) = 0$. The number $n$ is called \emph{the
  valency} of $H$ at $x$ and we denote it by $\val(H,x)$. The points in
$$
\mathcal C = \left\{x \in S : \ \val(H,x) \geq 2 \right\} = \left\{x \in S : H'(x)
  = 0\right\}
$$
are \emph{the critical points} of $H$.

\begin{lemma}\label{first} Let $H : S \to S$, $K : S \to S$ be
  non-constant holomorphic maps. Let $x,y \in S$ be points such that $H(x)
  = K(y)$. It follows that
there is a conformal germ $\eta \in \mathcal H_x$ such that 
\begin{itemize}
\item[a)]  $\eta(x) = y$ and
\item[b)] $K(\eta(z)) = H(z)$ for all $z$ in a neighbourhood of $x$
\end{itemize}
if and only if $\val(H,x) = \val(K,y)$, in which case there are
exactly $\val(H,x)$ elements of $\mathcal H_x$ with these properties,
up to equivalence of germs. Furthermore, two elements in $\mathcal
H_x$ which both satisfy a) and b) are equivalent if and only if they
have the same derivative at $x$.
 \end{lemma}
\begin{proof} Assume first that there is a conformal germ $\eta \in
  \mathcal H_x$ such that a) and b) hold. Then
$$
\val (H,x) = \val (K \circ \eta, x) = \val (K,y)\val(\eta,x) = \val (K,y),
$$
proving the necessity of the condition. Assume next that $\val(H,x) =
\val(K,y) = j$. When $j =1$, $H \in \mathcal H_x$
and $K \in \mathcal H_y$ and hence $\eta$ is the only
element of $\mathcal H_x$, up to equivalence, which satisfies a) and b). Assume $j \geq
2$. Working locally using local charts at $x$, $y$ and $K(y) = H(x)$ we may assume that $S = \Delta$ and that $x=y = 0$. It follows from B\"ottchers theorem, Theorem
9.1 in \cite{Mi}, that there are
conformal germs $\delta \in \mathcal H_0$ and $\mu \in \mathcal H_0$
such that $\delta(0) = \mu(0) = 0$ and
$$
K(\delta(z)) - K(0) = \delta\left(z^j\right)
$$
and
$$
H(\mu(z)) - H(0) = \mu\left(z^j\right)
$$
for all $z$ in a neighbourhood of $0$. Then $\eta \mapsto \delta^{-1} \circ \eta
\circ \mu$ is a bijection from the set of elements $\eta \in
\mathcal H_0$ which satisfy a) and b) onto the set of elements
$\kappa \in \mathcal H_0$ such that
\begin{itemize}
\item[a')] $\kappa(0) = 0$ and
\item[b')] $\kappa(z)^j = \delta^{-1} \circ \mu(z^j)$ for all $z$ in a neighbourhood of $0$.
\end{itemize}
The Taylor expansion of $z \mapsto \delta^{-1} \circ
\mu \left(z^j\right)$ at $0$ has the form
$z^j \sum_{n=0}^{\infty} b_n z^n$
where $b_0 = \left(\delta^{-1} \circ \mu\right)'(0) \neq 0$. By using a
holomorphic logarithm near $b_0$ we get a holomorphic function $\kappa_0$ such that $\kappa_0(z)^j =
\sum_{n=0}^{\infty} b_n z^n$ for all $z$ in a neighbourhood of
$0$. Set $\kappa(z) = z\kappa_0(z)$. Then $\kappa \in \mathcal H_0$
since $\kappa'(0) = \kappa_0(0) \neq 0$, and we have proved
existence. Since elements $\kappa$ of $\mathcal H_0$ that satisfy a') and b')
agree up to multiplication by a $j$th root of unity, we see that
their number is exactly $j$, and that their equivalence class in
$\mathcal H_0$ is
determined by their derivative at $0$. 
\end{proof}

By specialising to the case where $K = H$ and $x=y$ the preceding
proof also shows the following. 

\begin{lemma}\label{cyclicgroup} Let $H : S \to S$ be a non-constant
  holomorphic map and let $x\in S$. It follows that
the equivalence classes of conformal germs $\eta \in \mathcal H_x$ which satisfy that
\begin{enumerate}
\item[i)] $\eta(x)= x$ and
\item[ii)] $H(\eta(z)) = H(z)$ for all $z$ in a neighbourhood of $x$ in $S$ 
\end{enumerate}
form a cyclic group of order $\val(H,x)$ under composition.
\end{lemma}

\begin{prop}\label{etale!!}$G_X$ has a countable base for its
  topology, i.e. $G_X$ is a second countable \'etale locally compact Hausdorff groupoid. 
\end{prop}
\begin{proof}
To construct a countable base for the topology of $G_X$, fix $n,m \in
\mathbb N$. Since $H$ is not constant the set of points that are
critical for either $H^n$ or $H^m$ is a discrete subset of
$S$. Furthermore, $S$ is second countable by \S 2 in
\cite{Mi}. There is therefore a base
$\{U_i\}_{i=1}^{\infty}$ for the topology of $S$ such that each
$U_i$ is connected and contains at most one critical point of $H^n$ or
$H^m$. Furthermore, we can arrange that $H^m$ is injective on $U_j$
unless there is critical point of $H^m$ in $U_j$. Let $i,j \in \mathbb N$. When $U_i$
contains no critical points of $H^n$ and $U_j$ no critical point of
$H^m$ it follows that there is at most
one conformal germ $\mu$ with $U_i$ as domain of
definition and satisfying that 
\begin{equation}\label{germeq}
  \mu(U_i)
\subseteq U_j \ \text{and} \    H^n(z) = H^m(\mu(z)) \ \forall z \in U_i.
\end{equation} 
When $U_i$ contains a critical point $x_i$
of $H^n$ it follows from Lemma \ref{first} that for each $j \in
\mathbb N$ the number of
conformal germs $\eta$ with $U_i$ as domain such that (\ref{germeq}) holds is at most $\val \left(H^n,x_i\right)$. Hence
the collection $A(n,m,i,j)$ of conformal germs $\mu$ with the properties
that $\mu$ is defined on $U_i$
such that (\ref{germeq}) holds, and $U_j$ does not
contain a critical point of $H^m$ when $U_i$ does not contain one of
$H^n$, is finite. By definition of the
pseudo-group $\mathcal P$, and our assumptions on the subset $X$, it follows that for any $\eta \in \mathcal T_{k}\left(H|_X\right)$
the set $U(\eta)$ from (\ref{baseset}) is a union of sets of the form 
$$
\left\{\left[x,k,\mu|_X, \mu(x)\right] :  x \in U_i \cap X \right\}
$$ 
for some $(n,m,i,j)$ and some $\mu \in A(n,m,i,j)$. Such sets therefore form a countable base
for the topology of $G_X$.
\end{proof}

\subsection{The isotropy groups}

We say that a point $x \in X$ is \emph{pre-periodic} when there is a
$p \in \mathbb N \backslash \{0\}$ such that $H^{p+n}(x) = H^n(x)$ for
all large $n$ and that $x$ is \emph{pre-critical} when there is an $n
\in \mathbb N \cup \{0\}$ such that $H^n(x)$ is critical for $H$.

\begin{prop}\label{isotropy} Let $x \in X$ and let
  $\left(G_X\right)^x_x$ be the isotropy group at $x$ in $G_X$.

\begin{itemize}
\item[a)] When $x$ is neither pre-critical nor pre-periodic,
  $\left(G_X\right)_x^x \simeq 0$.
\item[b)] When $x$ is pre-periodic and not pre-critical,
  $\left(G_X\right)_x^x \simeq \mathbb Z$.
\item[c)] When $x$ is pre-critical and not pre-periodic
  $\left(G_X\right)_x^x$ is isomorphic to a non-zero subgroup of $\mathbb
  Q/\mathbb Z$.
\item[d)] Assume that $x$ is both pre-periodic and pre-critical. Let $n \in
  \mathbb N$ be the least number such that $H^n(x)$ is periodic.

\begin{itemize}
\item[d1)] When $H^j(x) \in \mathcal C$ for some $j \geq n$ the isotropy group
  $\left(G_X\right)_x^x$ is isomorphic to an infinite subgroup of $\mathbb
  Q/\mathbb Z$, and 
\item[d2)] when $H^j(x) \notin \mathcal C$ for all $j \geq n$ the isotropy group
  $\left(G_X\right)_x^x$ is isomorphic to $\mathbb Z \oplus
  \mathbb Z_d$ where $d = \val(H^n,x)$.
\end{itemize}
\end{itemize}
\end{prop}

\begin{proof} a) Assume that $[x,k,\eta,x] \in \left(G_X\right)^x_x$ is not in the
unit space. If $k \neq 0$ the point $x$ is pre-periodic. If
$k = 0$ and $\left[\eta\right]_x \neq [\id]_x$ it follows from Lemma \ref{cyclicgroup}
that $2 \leq \val (H^n,x) = \val(H, H^{n-1}(x))\val(H,H^{n-2}(x)) \cdots \val(H,x)$ for all
large $n$; that is, $x$ is pre-critical.

b) Let $p \in \mathbb N$ be the
least natural number such that $H^{p+n}(x) = H^n(x)$ for all large
$n$. If $\eta$ is a holomorphic germ at $x$ such that $H^{p+n}(z) =
H^n(\eta(z))$ for all $z$ in a neighbourhood of $x$, it follows that
$H^n \in \mathcal H_x$ since $\val(H^n,x) =\val (H^{n+p},x) = 1$ and hence that 
$\eta = H^{-n}\circ H^{p+n}$ in $\mathcal H_x$. Thus
$$
\left(G_X\right)^x_x = \left\{ \left[x,kp, \eta^k,x\right] : \ k \in \mathbb
  Z  \right\} \ \simeq \ \mathbb Z .
$$

c)  Since $x$ is not pre-periodic
\begin{equation}\label{lur}
\left(G_X\right)_x^x = \bigcup_n \left\{ [x,0,\eta,x] : \ \eta \in
  \mathcal H_x(n) \right\} 
\end{equation}
where $\mathcal H_x(n)$ is the set 
\begin{equation*}
\left\{ \eta \in \mathcal H_x : \ \eta(x) = x \
  \text{and} \ H^n(\eta(z)) = H^n(z) \ \text{for all $z$ in a
    neighbourhood of $x$} \right\} .
\end{equation*}
It follows from Lemma \ref{cyclicgroup} that $\mathcal H_x(n)$ is a
finite cyclic group which is non-zero for all large $n$ since $x$ is
pre-critical. Hence $\left(G_X\right)^x_x$ is a non-zero union of
an increasing sequence of finite cyclic groups and hence a non-zero
subgroup of $\mathbb Q/\mathbb Z$. 

d1)  Let
$p$ be the period of $H^n(x)$. Let $[x,k,\eta,x]  \in G_X$ and assume that $k \neq 0$. Then 
$k = n_0p$ for some $n_0\in \mathbb Z \backslash \{0\}$. Let $l \in
\mathbb N$ be so large that $n_0 + l > 0$. Let $\eta \in \mathcal T^k\left(H|_X\right)$
satisfy $\eta(x) = x$ and
$H^{n_0p + lp + n}(z) = H^{lp +n}\left(\eta(z)\right)$ for all $z$ in
an neighbourhood of $x$ and some $n \in \mathbb N$. Then
\begin{equation*}
\begin{split}
& \val\left(H^p,H^n(x)\right)^{n_0 +
  l}\val(H^n,x) = \val(H^{n_0p + lp + n},x) \\
& = \val\left(H^{lp +n},x\right) = \val\left(H^p,
  H^n(x)\right)^l \val\left(H^n,x\right)
\end{split}
\end{equation*} 
and we conclude that
$$
\val\left(H^p,H^n(x)\right)^{n_0 +
  l} =  \val\left(H^p,
  H^n(x)\right)^l .
$$
Since $n_0 \neq 0$ this implies that
$\val\left(H^p,H^n(x)\right) = 1$ and hence by periodicity that $\val\left(H,
  H^j(x)\right) = 1$ for all $j \geq n$. This contradicts that we are
in case d1) and we conclude
therefore that $k = 0$ when $[x,k,\eta,x]  \in G_X$, i.e. the equality
(\ref{lur}) holds. To see that $\left(G_X\right)_x^x$ is infinite note that there is an $m_0\in
\mathbb N$ such that $H^{m_0}(x)$ is $p$-periodic and
$\val\left(H,H^{m_0}(x)\right) \geq 2$. Then $\val\left(H^{lp},H^{m_0}(x)\right) \geq 2^l$ for all $l \in \mathbb N$
and hence $\val\left( H^{m_0+lp},x \right) \geq 2^l$ for all $l
\in \mathbb N$. It follows that $\lim_{n \to \infty} \# \mathcal H_x(n) = \infty$. 

d2) In this case $d = \val(H^n,x) \geq 2$, and by Lemma
\ref{cyclicgroup} the elements $\eta \in \mathcal H_x$ which satisfy
that $\eta(x) = x$ and that $H^n(\eta(z)) = H^n(z)$ is a neighbourhood
of $x$ form a cyclic group $F$ of order $d$. By assumption
$\val\left(H^p,H^n(x)\right) = 1$; i.e. 
$H^p \in \mathcal H_{H^n(x)}$. Since $\val \left(
  \left(H^p\right)^{-1} \circ H^n,x\right) = \val\left(H^n,x\right)$, it
  follows from Lemma \ref{first} that there is a conformal germ $\xi$
  at $x$ such that $\xi(x) = x$ and $\left(H^p\right)^{-1} \circ
  H^n(z) = H^n(\xi(z))$ in a neighbourhood of $x$. Then
$$
\left[x,kp, \mu \circ \xi^{-k} ,x\right] \in \left(G_X\right)_x^x 
$$
for all $k \in \mathbb Z$ and all $\mu \in F$. When $[x,kp,\mu, x],
[x,kp,\mu_1,x] \in \left(G_X\right)^x_x$ we have that $H^n(\mu(z)) =
H^n(\mu_1(z))$ in a neighbourhood of $x$ and hence $\mu^{-1}\circ \mu_1
\in F$. Thus
$$
\left(G_X\right)_x^x =  \left\{ \left[x,kp, \mu \circ \xi^{-k}
    ,x\right] : \ k \in \mathbb Z, \ \mu \in F \right\} 
$$
which is clearly an extension of $\mathbb Z$ by $\mathbb Z_d$. It
follows from the chain rule and the last statement of Lemma \ref{first}
that $\left[\xi \circ \mu \circ \xi^{-1}\right]_x = [\mu]_x$ for all
$\mu \in F$. This shows that $\left(G_X\right)_x^x$ is abelian
and therefore isomorphic to $\mathbb Z \oplus \mathbb Z_d$.
\end{proof}






\section{The conformal action and its KMS states}

\subsection{The conformal action}\label{confac}

Consider $S$ as a 1-dimensional complex manifold with a given
metric $g$, by which we mean that $g$ is a continuous choice of norms on the tangent spaces.
By definition of $G_X$ we can then define a map $L : G_X \to \mathbb R$
such that
\begin{equation}\label{Ldef}
L[x,k,\eta,y] = \log \left|\eta'(x)\right|_g
\end{equation}
where $\eta'$ denotes the differential of $\eta$ and
$\left|\eta'(x)\right|_g$ the norm of $\eta'(x)$ calculated with respect
to the metric $g$. Then $L[x,k,\eta,y][y,l,\mu,z] = L[x,k,\eta,y] +
L[y,l,\mu, z]$,
i.e. $L$ is a homomorphism $L : G_X \to \mathbb R$ which is obviously
continuous. The corresponding one-parameter group $\gamma
=\left(\gamma_t\right)$ of
$C^*_r\left(G_X\right)$ is determined by the condition that
\begin{equation}\label{Lform}
\gamma_t(f)[x,k,\eta,y] = \left|\eta'(x)\right|_g^{it} f[x,k,\eta, y]
\end{equation}
when $f \in C_c\left(G_X\right)$. We call $\gamma$ \emph{the conformal
  action}. 

By construction the conformal action depends on the chosen 
metric. However, the choice of another metric does not change the
structure of the KMS states when $X$ is compact. To see
this let $g^1$ be another metric on $S$. There is then a strictly positive
function $r : S \to ]0,\infty[$ such that $g^1 = rg$. It follows that when we use $g^1$ instead of $g$
in (\ref{Ldef}) we get a continuous homomorphism $L^1 : G_X \to
\mathbb R$ such that
\begin{equation}\label{LL1}
L^1[x,k,\eta,y] = L[x,k,\eta,y] + \log r(y) - \log r(x).
\end{equation}
Let $\gamma^1$ be the automorphism group obtained from $L^1$. Set
$u_t(x) = e^{-it \log r(x)}$. Then $\{u_t\}_{t \in \mathbb R}$ is a strictly continuous unitary
one-parameter group in $M(C^*_r(G_X))$ - the multiplier algebra of
$C^*_r\left(G_X\right)$ - and it is easy to check that
\begin{equation}\label{innerrel}
u_t\gamma_t(a)u_t^* = \gamma^1_t(a)
\end{equation}
for all $t $ and $a$. Note that the $u_t$'s are in the fixed
point algebra of both $\gamma$ and $\gamma^1$. In particular, it follows therefore
from (\ref{innerrel}) that $\gamma$ and $\gamma^1$ are
exterior equivalent, cf. 8.11.3 of \cite{Pe}. It also gives the following.

\begin{lemma}\label{metricdep} Assume that $X$ is compact. Let
  $\beta \in \mathbb R$. There is a
  bijective map from the $\beta$-KMS states of $\gamma$ onto the $\beta$-KMS
  states of $\gamma^1$.
\end{lemma}
\begin{proof} Since $X$ is compact, $r$ is bounded away from
  both $0$ and $\infty$. Then $h(x) = e^{-\frac{ \beta}{2} \log r(x)} \in
 C^*_r(G_X)$. When $\omega$ is a $\beta$-KMS state for $\gamma$
  the state
$$
a \mapsto \frac{\omega(hah)}{\omega(h^2)} 
$$
will be $\beta$-KMS state for $\gamma^1$, giving us a bijection with
the obvious inverse. 
\end{proof}

When $X$ is not compact the conformal action and its KMS states
depends very much on the metric. See Remark \ref{flat}.


\subsection{The KMS states of the conformal action}\label{KMSconf7}

It follows from Theorem \ref{decomp!} that the $\beta$-KMS states for
the conformal action can be described in terms of non-atomic
$(G_X,L)$-conformal measures and atomic measures supported on consistent
and $\beta$-summable $G_X$-orbits, as defined in Section \ref{A}. We describe in this section what these notions become in the
present setting and formulate the according version of Theorem
\ref{decomp!}. 

Let $\beta \in \mathbb R \backslash \{0\}$. Note that a finite Borel measure $m$ on $X$
is $(G_X,L)$-conformal with
    exponent $\beta$, as defined in Section \ref{A}, when the equality
\begin{equation}\label{sconformal}
m(\eta(U)) = \int_U \left|\eta'(x)\right|_g^{\beta} dm(x)
\end{equation}
holds for all local transfers $\eta : U \to V$ for $H|_X$. We shall
compare this to the more conventional notion of conformality which
goes back to the work of Sullivan in \cite{S}. Thus a finite Borel
measure $m$ on $X$ is \emph{$\beta$-conformal} when
\begin{equation}\label{confequa}
m(H(A)) = \int_A \left|H'(x)\right|_g^{\beta} \ dm(x)
\end{equation}
for every Borel subset $A \subseteq X$ for which $H : A \to X$ is
injective.

\begin{lemma}\label{n-mlemma4} Let $m$ be a $\beta$-conformal
  measure. Let $U,V$ be open subsets of $X$ and $n,l \in \mathbb N$
  natural numbers such that $H^n(U) = H^l(V)$. Assume that $H^n$ and $H^l$ are injective on $U$ and $V$, respectively. Then
\begin{equation}\label{desire} 
m(V) = m(H^{-l}(H^n(U)) = \int_U \left|\left( H^{-l} \circ H^n\right)'(x)\right|_g^{\beta} \ d m(x)
\end{equation}
where $H^{-l}$ denotes the inverse of $H^l : V \to H^l(V)$.
\end{lemma}
\begin{proof} By regularity of the Borel measures on $U$ given by $B
  \mapsto m(H^{-l}(H^n(B))$ and $B \mapsto \int_B \left|\left( H^{-l}
      \circ H^n\right)'(x)\right|_g^{\beta} \ d m(x)$ it suffices to
  prove that 
\begin{equation}\label{desire2}
m(H^{-l}(H^n(K)) = \int_K \left|\left( H^{-l} \circ
      H^n\right)'(x)\right|_g^{\beta} \ d m(x)
\end{equation}
when $K \subseteq U$ is compact.

 In the following we write $t \overset{\delta}{\sim} s$
  when $s,t, \delta \in \mathbb R, \ \delta > 0$ and $\left|s-t\right|
  \leq \delta$, and the inverses $H^{-1}$ which occur all come from
  the branch 
\begin{equation*}
\begin{xymatrix}{
H^l(V) \ar[r]^-{H^{-1}} & H^{l-1} (V) \ar[r]^-{H^{-1}} & H^{l-2} (V)
\ar[r]^-{H^{-1}} & \cdots \cdots \ar[r]^-{H^{-1}} & V . }
\end{xymatrix}
\end{equation*} 
Let $\epsilon
  \in ]0,1[$. We can then find a finite Borel partition $K = \sqcup_i K_i$ of
  $K$ such that $\left|H'(a)\right|_g^{\beta}
\overset{\epsilon}{\sim} \left|H'(b)\right|_g^{\beta}$ when $a,b \in
H^{-k}(H^n(K_i)), \ 0 \leq k \leq l$, or $a,b \in H^j(K_i), \ 0 \leq
j \leq n$.
Furthermore, we can arrange that 
\begin{equation}\label{altso}
\left|\left(H^{-l} \circ H^n\right)'(x)\right|_g^{\beta}
    \overset{\epsilon}{\sim}  \left|\left(H^{-l} \circ
    H^n\right)'(y)\right|_g^{\beta}
\end{equation}
when $x,y \in K_i$. For each $i$ we choose a
point $a_i \in K_i$. Set
$$
L = \sup \left\{ \left|H'(x)\right|_g^{-\beta} + \left|H'(x)\right|_g^{\beta} : \
  x \in \bigcup_{0 \leq i \leq l} H^{-i}(H^n(K)) \cup \bigcup_{0 \leq
    j \leq n} H^i(K) \right\}.
$$

Consider now an arbitrary $i$ and set $\epsilon_1 = \epsilon
m\left(H^{-l}(H^n(K_i))\right)$. It follows then from
  (\ref{confequa}) that
\begin{equation}\label{ext11}
m\left(
  H^{-l}\left(H^n(K_i)\right)\right)\left|H'\left(H^{-l}\left(H^n(a_i)\right)\right)\right|_g^{\beta} \overset{\epsilon_1}{\sim} m\left(H^{-l+1}\left(H^n(K_i)\right)\right) 
\end{equation}
and hence, since $\left|H'\left(H^{-l}\left(H^n(a_i)\right)\right)\right|_g =
\left|(H^{-1})'\left(H^{-l+1}\left(H^n(a_i)\right)\right)\right|_g^{-1}$, that
$$
m\left(
  H^{-l}\left(H^n(K_i)\right)\right) \overset{L \epsilon_1}{\sim} \left|(H^{-1})'\left(H^{-l+1}\left(H^n(a_i)\right)\right)\right|_g^{\beta}m\left(H^{-l+1}\left(H^n(K_i)\right)\right) .
$$
It follows from (\ref{ext11})
that 
$$
m\left(H^{-l+1}(H^n(K_i))\right) \leq (L + 1)
m\left(H^{-l}(H^n(K_i))\right).
$$
In the same way we get the
estimates
$$
m\left(H^{-l+2}(H^n(K_i))\right) \leq (L + 1)
m\left(H^{-l+1}(H^n(K_i))\right),
$$
and
$$
m\left(
  H^{-l+1}\left(H^n(K_i)\right)\right) \overset{L \epsilon_2}{\sim} \left|(H^{-1})'\left(H^{-l+2}\left(H^n(a_i)\right)\right)\right|_g^{\beta}m\left(H^{-l+2}\left(H^n(K_i)\right)\right) 
$$
where $\epsilon_2 = \epsilon
m\left(H^{-l+1}(H^n(K_i))\right)$. After $l$ repetitions of these
arguments we find that 
\begin{equation}\label{firststep}
\begin{split}
& m\left(
  H^{-l}\left(H^n(K_i)\right)\right)  \overset{\epsilon'}{\sim}   \prod_{j=0}^{l-1} \left|(H^{-1})'\left(
    H^{-j}\left(H^n(a_i)\right)\right)\right|_g^{\beta} m\left(
  H^n(K_i)\right) \\
& = \left|\left(H^{-l}\right)'\left(H^n(a_i)\right)\right|_g^{\beta} m(H^n(K_i))
\end{split}
\end{equation}
where $\epsilon' = \epsilon lL^l(L+1)^lm\left(
  H^{-l}\left(H^n(K_i)\right)\right)$.

Now attack $H^n(K_i)$ in a similar way. It follows from (\ref{confequa}) that
$$
m(H^n(K_i)) \overset{\epsilon_4}{\sim}
\left|H'(H^{n-1}(a_i))\right|_g^{\beta} m(H^{n-1}(K_i))
$$
where $\epsilon_4 = \epsilon  m(H^{n-1}(K_i))$. In particular,
$$
m(H^n(K_i)) \leq (L+1)m(H^{n-1}(K_i)) .
$$
Similarly,
$$
m(H^{n-1}(K_i)) \overset{\epsilon_5}{\sim}
\left|H'(H^{n-2}(a_i))\right|_g^{\beta} m(H^{n-2}(K_i)),
$$
where $\epsilon_5 =  \epsilon m(H^{n-2}(K_i))$,
and
$$
m(H^{n-1}(K_i)) \leq (L+1)m(H^{n-2}(K_i)) .
$$
After $n$ steps we find that
$$
m(H^n(K_i)) \overset{\epsilon''}{\sim}
\prod_{j=0}^{n-1} \left|H'(H^{j}(a_i))\right|_g^{\beta} m(K_i) =
\left|(H^n)'(a_i)\right|_g^{\beta} m(K_i)
$$
where $\epsilon'' = \epsilon nL^n(L+1)^nm(K_i)$. Combining with
(\ref{firststep}) we find that
\begin{equation}\label{tot-est}
\begin{split}
&m\left(  H^{-l}\left(H^n(K_i)\right)\right) \overset{\delta}{\sim} 
 \left|\left(H^{-l}\right)'\left(H^n(a_i)\right)\right|_g^{\beta}
 \left|(H^n)'(a_i)\right|_g^{\beta} m(K_i) \\
& = \left|\left(H^{-l} \circ H^n\right)'(a_i)\right|_g^{\beta} m(K_i)
\end{split}
\end{equation}
where $\delta = \epsilon' + L^l \epsilon''$. By summing over $i$ this
gives us the estimate
$$
m\left(H^{-l}\left(H^n(K)\right)\right) \overset{\delta'}{\sim} \sum_i
  \left|\left(H^{-l} \circ H^n\right)'(a_i)\right|_g^{\beta} m(K_i) 
$$
where $\delta' = \epsilon \left(lL^l(L+1)^l + L^l
  nL^n(L+1)^n\right)$. It follows from (\ref{altso}) that
$$
\sum_i
  \left|\left(H^{-l} \circ H^n\right)'(a_i)\right|_g^{\beta} m(K_i)
  \overset{\epsilon}{\sim} \int_K \left|\left(H^{-l} \circ
      H^n\right)'(x)\right|_g d m(x),
$$
which gives us (\ref{desire2}) because $\epsilon > 0$ was arbitrary. 
\end{proof}

Let $\mathcal C$ be the set of critical points of $H$ contained in $X$.

\begin{lemma}\label{trans=conf4} Let $m$ be a finite Borel measure on $X$.
\begin{enumerate}
\item[a)] Assume that $m$ is $(G_X,L)$-conformal with exponent $\beta$ and that
  $m(H(\mathcal C)) = 0$. It follows that $m$ is
  $\beta$-conformal.
\item[b)] Assume that $m$ is $\beta$-conformal and that
  $m(\mathcal C) = 0$. It follows that $m$ is a $(G_X,L)$-conformal
  with exponent ${\beta}$.
\end{enumerate}
\end{lemma}

\begin{proof} a) Let $A \subseteq X$ be a Borel subset such that $H$
  is injective on $A$. Since $H(A) \backslash H(A \cap \mathcal C) = H(A
  \backslash \mathcal C)$ it follows that $m(H(A)) = m(H(A \backslash
  \mathcal C))$ beacuse $m(H(\mathcal C)) =0$ by assumption. Write $X \backslash \mathcal C = \bigcup_{i=1}^{\infty} U_i \cap X$ where $U_i$
  is open in $S$ and $H$ is univalent on $U_i$ for all $i$. There is a partition $A \backslash \mathcal C =
  \sqcup_{i=1}^{\infty} A_i$ of $A \backslash \mathcal C$ into Borel
  sets such that $A_i \subseteq U_i$. Let $V \subseteq U_i$ be an open
  subset. Since $H$ is univalent on $U_i$ and $H^{-1}\left(X\right)
  = X$ it follows that $H : V \cap X \to H(V) \cap X$ is a local
  transfer for $H|_X$ and hence that
$$
m\left(H(V \cap X)\right) = \int_{V\cap X} \left|H'(z)\right|_g^{\beta} dm(z)
$$
since $m$ is $(G_X,L)$-conformal. The regularity of the Borel measures
on $U_i \cap X$ given by
$D \mapsto m(H(D))$ and $D \mapsto \int_D  \left|H'(z)\right|_g^{\beta}
dm(z)$ now implies that $m\left(H(A_i)\right) = \int_{A_i}
\left|H'(z)\right|_g^{\beta} dm(z)$. Summing over $i$ we find that
$$
m(H(A)) = m(H(A \backslash \mathcal C)) = \int_{A \backslash \mathcal C}
\left|H'(z)\right|_g^{\beta} \ dm(z) = \int_A 
\left|H'(z)\right|_g^{\beta} \ dm(z).
$$ 
 
b) Since $m$ is $\beta$-conformal with no atoms at points in $\mathcal C$
it follows that 
\begin{equation}\label{more}
m\left( \bigcup_{j \in \mathbb N} H^{-j}(\mathcal C)\right) = 0 .
\end{equation} 
Let $\eta : U \to V$ be a local transfer for $H|_X$ and let $n,l \in
  \mathbb N$ be such that $H^n(z) = H^l(\eta(z))$ for all $z \in
  U$. To establish the equality (\ref{sconformal}) we may assume that
  $U = X \cap U_0$ where $U_0 \subseteq S$ is open and relatively
  compact in $S$. Set $E = U \cap \bigcup_{0 \leq j \leq n}
  H^{-j}(\mathcal C)$ which is then a finite set. Let $W$ be an open subset
  of $U \backslash E$ such that $H^n|_W$ is injective. Then $H^l$ is
  injective on $\eta(W)$ and $\eta = H^{-l} \circ H^n$ where $H^{-l}$
  denotes the inverse of $H^l : \eta(W) \to H^n(W)$. It follows
  therefore from Lemma \ref{n-mlemma4} that  
\begin{equation}\label{equal}
m(\eta( W)) = \int_{W}
\left|\eta'(x)\right|_g^{\beta} \ d m(x).
\end{equation}
The same equality holds, for the same reason, for any open subset of
$W$ and therefore the regularity of the Borel measures on $W$ given by $B \mapsto m(\eta(B))$
and $B \mapsto \int_{B}
\left|\eta'(x)\right|_g^{\beta} \ d m(x)$ shows that
(\ref{equal}) holds when $W$ is substituted by any Borel subset of
$W$. Since $U\backslash E$ is the union of a countable partition of Borel sets each
of which is a subset of an open set on which $H^n$ is injective, we
conclude that (\ref{equal}) holds when $W$ is replaced by $U
\backslash E$. Since $\eta(E)$ and $E$ are both subsets of $\bigcup_{j
  \geq 0} H^{-j}(\mathcal C) $ it follows from (\ref{more}) that $m(V) = \int_{U}
\left|\eta'(x)\right|_g^{\beta} \ dm(x)$.
\end{proof}

It follows from Lemma \ref{trans=conf4} that except for measures with an
atom at a critical point or a critical value, the conformal measures are
the same as the $(G_X,L)$-conformal measures. This
conclusion can not be improved; there are examples with purely atomic
$(G_X,L)$-conformal measures that are not conformal, and also examples 
with purely atomic conformal measures that are not
$(G_X,L)$-conformal. See Remark \ref{example}. However, it
follows from Lemma \ref{trans=conf4} that the two notions of
conformality agree for non-atomic measures, and this conclusion
suffices for the present purposes.

We turn next to a description of the consistent $G_X$-orbits.

\begin{lemma}\label{pass} Let $K : S \to S$ be a holomorphic map and
  $x \in S$ a point such that $\val(K,x) = j \geq 2$. Let $\kappa \in
  \mathcal H_{K(x)}$ such that $\kappa(K(x)) = K(x)$. There is then a conformal germ $\delta \in
  \mathcal H_x$ such that $\kappa \circ K(z) = K \circ \delta(z)$ for
  all $z$ in a
  neighbourhood of $x$, $\delta(x) =x$ and $\delta'(x)^j = \kappa'(K(x))$.
\end{lemma}
\begin{proof} We prove this first when $S$ is an open neighbourhood of
  $0$ in $\mathbb C$, $x = 0$ and $K(0) = 0$. In this case it follows
  from B\"ottcher's theorem that there is conformal germ $\varphi \in
  \mathcal H_0$ such that $\varphi(0) = 0$ and $\varphi \circ K \circ
  \varphi^{-1}(z) = z^j$ in a neighbourhood of $0$. The Taylor series
  of $\varphi \circ \kappa \circ \varphi^{-1}(z^j)$ at $0$ has the
  form $z^j\left( \sum_{n=0}^{\infty}b_nz^n\right)$ with $b_0 \neq
  0$. Using a holomorphic logarithm near $b_0$ we get a holomorphic
  map $\delta_1$ such that $\delta_1(z)^j =
  \sum_{n=0}^{\infty}b_nz^n$ in a neighbourhood of $0$. Set $\delta_2(z) = z\delta_1(z)$. Then
  $\delta_2 \in \mathcal H_0$, $\delta_2(z)^j = \varphi \circ \kappa
  \circ \varphi^{-1}(z^j)$ in a neighbourhood of $0$ and $\delta_2'(0)^j = \delta_1(0)^j =
  b_0 =
  \left(\varphi \circ \kappa \circ \varphi^{-1}\right)'(0) =
  \kappa'(0)$. Set $\delta = \varphi^{-1}\circ \delta_2 \circ \varphi$.  

The general case: There are local charts $\varphi, \psi$ defined in an open
neighbourhood of $x$ and $K(x)$, respectively, such that $\psi(K(x)) =
0$ and $\varphi(x) = 0$. It follows from the case dealt with above
that there is a conformal germ $\delta_3$ at
$0$ such that $\delta_3'(0)^j = \left(\psi \circ \kappa \circ
  \psi^{-1}\right)'(0) = \kappa'(K(x))$ and $\psi \circ \kappa \circ \psi^{-1} \circ \psi \circ
K \circ \varphi^{-1}(z) = \psi \circ K \circ \varphi^{-1}\circ
\delta_3(z)$ for all $z$ in a neighbourhood of $0$. Then $\delta =
\varphi^{-1} \circ \delta_3 \circ \varphi$ will have the required
properties.

\end{proof}

Recall that the orbit of a periodic point $y$ of period $p$ is \emph{neutral} when
$\left|(H^p)'(y)\right|_g =1$ and
\emph{critical} when $(H^p)'(y) = 0$.

\begin{lemma}\label{consistent}
Let $x \in X$. The $G_X$-orbit $G_Xx$ is consistent if and only if $x$ is
either not pre-periodic or is pre-periodic to a critical or neutral
periodic orbit. 

\end{lemma}
\begin{proof} We must show that $L\left(\left(G_X\right)^x_x\right) \neq \{0\}$ if
and only if $x$ is pre-periodic to a periodic orbit which is neither
critical nor neutral. It follows from Proposition \ref{isotropy} that
  $\left(G_X\right)^x_x$ is a torsion group unless $x$ is pre-periodic to a non-critical periodic orbit. So what remains is to
  prove that when $x$ is pre-periodic to a non-critical periodic
  orbit, $L\left(\left(G_X\right)^x_x\right) = \{0\}$ if and only if
  that periodic orbit is neutral. Assume therefore that $x$ is
  pre-periodic to non-critical periodic orbit. Let $n \in \mathbb
  N$ be the least natural number such that $H^n(x)$ is periodic, say
  of period $p$, and
  assume that
  $\val(H,H^j(x)) = 1$ for all $j \geq n$. Let $\eta$ be a local
  transfer such that $\eta(x) = x$ and $H^{kp +m}(z) = H^m(\eta(z))$ in a neighborhood of $x$
  for some $k \in \mathbb N$ and some $m \geq n$. Then
  $\left(H^{kp}\right)'\left(H^m(x)\right)\left(H^m\right)'(x) =
  (H^m)'(\eta(x))\eta'(x)$ which implies that $\left|\eta'(x)\right|_g =
   \left|\left(H^{kp}\right)'\left(H^m(x)\right)\right|_g$ when
  $\left(H^m\right)'(x) \neq 0$. So when $x$ is not pre-critical we
  conclude that $L\left(\left(G_X\right)^x_x\right) = \{0\}$ if
  and only if $x$ is pre-neutral. 

It remains to consider the case when $\val(H^n,x) \geq 2$. It follows
from Lemma \ref{pass} that there is a conformal germ $\delta \in
\mathcal H_x$ such that $H^{kp + m}(z) = H^m\left(\delta(z)\right)$ in
a neighbourhood of $x$, $\delta(x) = x$ and $\delta'(x)^j =
\left(H^{kp}\right)'\left(H^m(x)\right)$ where $j = \val(H^m,x) =
\val(H^n,x)$. Then $H^m \circ \eta^{-1} \circ \delta(z) = H^m(z)$ for all
$z$ in a neighbourhood of $x$ and it follows from Lemma
\ref{cyclicgroup} that $\left|\left(\eta^{-1} \circ
    \delta\right)'(x)\right|_g = 1$. Hence $\left|\eta'(x)\right|_g^j =
\left|\delta'(x)\right|_g^j =
\left|(H^{kp})'\left(H^m(x)\right)\right|_g$. It follows that
$\left|\eta'(x)\right|_g = 1$ if and only if $x$ is pre-neutral.   
\end{proof}

 We summarise now the results on the KMS states for the conformal
 action which we obtain by specialising Theorem \ref{decomp!} to $G_X$. For $x \in X$ let $G_Xx$ be the $G_X$-orbit of $x$, i.e.
$G_Xx = \left\{ r(\xi) : \ \xi \in \left(G_X\right)_x \right\}$. 
Assume that $x$ is either not pre-periodic
  or is pre-periodic to a critical or
  neutral orbit. Then $G_Xx$ is consistent in the sense of
  Section \ref{A} by Lemma \ref{consistent}. For each $z \in G_Xx$ we
choose a local transfer $\eta$ such that $\eta(x) = z$ and set
$l_x(z) = \left|\eta'(x)\right|_g$. As in Section \ref{A} we will say that $G_Xx$ is \emph{$\beta$-summable} when $\sum_{z \in
  G_Xx} l_x(z)^{\beta} < \infty$. Assume that this is the case and let
$\varphi$ be a state of $C^*\left(G^x_x\right)$. For each $z \in \mathcal
O$ choose an element $\xi_z \in G_X$ such that $r(\xi_z) = z$ and
$s\left(\xi_z\right) = x$. Then $\xi_z
\left(G_X\right)_x^x\xi_z^{-1} =
\left(G_X\right)^z_z$. Define a state $\varphi_z$ on
$C^*\left(\left(G_X\right)^z_z\right)$ such that
$\varphi_z(f) = \varphi\left(f^{\xi_z}\right)$, 
where $f^{\xi_z}(y) = f\left(\xi_zy\xi_z^{-1}\right)$. As in Section
\ref{A} we define a state $\omega_{\mathcal O}^{\varphi}$ on $C^*_r\left(G_X\right)$ such that
$$
\omega_{\mathcal O}^{\varphi}(f) = \left(\sum_{z \in \mathcal O}
  l_x(z)^{\beta}\right)^{-1}  \sum_{z \in \mathcal O} 
\varphi_z\left(f|_{\left(G_X\right)^z_z}\right)\left|l_x(z)\right|_g^{\beta}
$$ 
for all $f \in C_c(G_X)$. 

We can now combine Lemma \ref{trans=conf4} and Lemma \ref{consistent}
with Theorem \ref{decomp!} to get the following.

\begin{thm}\label{decompGX} Let $m$ be a non-atomic $\beta$-conformal probability
  measure and let $x_i, i \in I$, be a finite or countably infinite
  collection of points in $X$ each of which is either not pre-periodic
  or is pre-periodic to a critical or
  neutral orbit. Let $\mathcal O_i = G_Xx_i$ be the $G_X$-orbit of
  $x_i$ and assume that they are all $\beta$-summable. Finally,
let $t_0,t_i \in ]0,1], i \in I$, be numbers such that $t_0 + \sum_{i
  \in I} t_i = 1$.

There is then a $\beta$-KMS state $\omega$ for the conformal action on $C^*_r(G_X)$
such that
\begin{equation}\label{decmppGX}
\omega(a) = t_0 \int_{X} P(a) \ dm + \sum_{i \in I} t_i \omega_{\mathcal
  O_i}^{\varphi_i}(a)
\end{equation}
for all $a \in C^*_r\left(G_X\right)$. Conversely, any $\beta$-KMS state
$\omega$ for the conformal action on $C^*_r(G)$ admits a unique decomposition of the form (\ref{decmppGX}).
\end{thm}

\begin{remark}\label{flat} We can now
  show by example that the structure of the KMS states
  of the conformal action depends on the chosen metric when
  $X$ is not compact. 
  Let $S = \mathbb C, H(z) = e^z$ and $X = \mathbb C \backslash
  \{0\}$. When the conformal action on $C^*_r\left(G_X\right)$ is
defined by use of the usual 'flat' metric $\left| \cdot \right|$ there are no $\beta$-KMS
states for the conformal action. To see this note that there are no
critical points in this case so that $(G_X,L)$-conformality is the
same as ordinary conformality by Lemma \ref{trans=conf4}. By Theorem
\ref{decompGX} it suffices therefore to show that for any $\beta \neq 0$ there
can be no Borel probability measure $m$ on $\mathbb C \backslash
\{0\}$ such that
\begin{equation}\label{flaeq}
m(e^A) = \int_A \left|e^z\right|^{\beta} \ dm(z)
\end{equation}
for every Borel subset $A$ of $\mathbb C \backslash \{0\}$ on which
$e^z$ is injective. To see that this is so, assume that $m$ is such a
measure. Let $w \in \mathbb C \backslash \{0\}$ and choose
an element $z \in \mathbb C \backslash \{0\}$ such that $e^z =
w$. There is an open neighborhood $U$ of $z$ of diameter $<  \pi$ such that $e^z$ is
injective on $U$. Then $\left(U + 2\pi n i\right) \cap \left(U + 2 \pi
  k i\right) = \emptyset$ when $n \neq k$ in $\mathbb N$ and (\ref{flaeq}) yields the
conclusion that 
$$
N m(e^U) = \sum_{k=1}^N \int_{U + 2\pi ki} \left|e^z\right|^{\beta} \
dm(z) \leq  \sup_{z \in U} \left|e^z\right|^{\beta}  
$$
for all $N\in \mathbb N$, which implies that $m(e^U) = 0$. Thus every $w \in \mathbb
C \backslash \{0\}$ has an open
neighborhood of zero $m$-measure. This is impossible by regularity of $m$.

Consider instead the metric $g(z) = \sqrt{f(z)} |{\cdot}|$ where $f:  \mathbb C \backslash
\{0\} \to ]0,\infty[$ is any continuous function such that
$\int_{\mathbb C \backslash \{0\}} f(z) \ dz = 1$ when $dz$ is the
Lebesgue measure. Using this
metric the requirement of $\beta$-conformality for a Borel probability
measure $m$ on $\mathbb C \backslash \{0\}$ is that
 \begin{equation}\label{flaeq2}
m(e^A) = \int_A \left|e^z\right|^{\beta} f(e^z)^{\frac{\beta}{2}} f(z)^{-\frac{\beta}{2}} \ dm(z)
\end{equation}
for every Borel subset $A$ of $\mathbb C \backslash \{0\}$ on which
$e^z$ is injective. It is easy to check that (\ref{flaeq2}) is satisfied
when $\beta =2$ and  
$m = f(z) \ dz$. 
Thus the conformal action, defined with the metric $g(z) =\sqrt{f(z)}|{\cdot}|$ has a
non-atomic $2$-KMS state.
\end{remark}

\section{Phase transition with spontaneous symmetry breaking from quadratic maps}

We now specialise to the case where $S$ is the Riemann sphere
$\overline{\mathbb C}$ and $H$ is a rational map on $\overline{\mathbb
  C}$ of degree at least $2$. We change
the notation accordingly by setting $H = R$. In this case there are three obvious
candidates for the set $X$; namely, $\overline{\mathbb C}$, the Fatou set
$F_R$ and the Julia set $J_R$. In this section we consider some
examples with $X = J_R$ where we can give a complete description of the
KMS states for the conformal action on $C^*_r\left(G_{J_R}\right)$
thanks to results of Graczyk
and Smirnov, \cite{GS1}, \cite{GS2}.   

Up to conjugation by a conformal automorphism of $\overline{\mathbb
  C}$ any polynomial of degree $2$ on the Riemann sphere has the form
$$
R_c(z) = z^2  +c 
$$
for some $c \in \mathbb C$. Since $J_{R_c} \subseteq \mathbb C$ we can
assume that the metric defining the conformal action is the usual
'flat' metric $\left|\cdot\right|$. We say that $R_c$ satisfies \emph{the
  Collet-Eckmann condition} when there is a $C > 0$ and a $\lambda >
1$ such that
\begin{equation}\label{CE}
\left|\left(R_c^n\right)'(c)\right| \geq C \lambda^n
\end{equation}
for all $n \in \mathbb N$. For example $z^2-2$ satisfies the
Collet-Eckmann condition because $R_{-2}(c) = R_{-2}(-2) = 2$ is a repelling
fixed point for $R_{-2}$. It was shown by Benedicks and Carleson in \cite{BC} that
the set of the $a$'s in $(0,2]$ for which the polynomial $R_{-a}$
satisfies the Collet-Eckmann condition is of positive Lebesgue
measure.

\begin{lemma}\label{summ} Assume that $R_c$ satisfies the
  Collet-Eckmann condition and let
  $HD\left(J_{R_c}\right)$ be the Hausdorff dimension of the Julia set
  $J_{R_c}$. 
\begin{enumerate}
\item[a)] When the critical point $0$ is
  pre-periodic there are no $\beta$-summable $G_{J_{R_c}}$-orbits for
  any $\beta  > 0$. 
\item[b)] When $0$ is not pre-periodic there are no $\beta$-summable
  $G_{J_{R_c}}$-orbits when $0 < \beta \leq HD\left(J_{R_c}\right)$, and when
  $\beta > HD\left(J_{R_c}\right)$, 
$$
G_{J_{R_c}}0 = \bigcup_{n \geq 0} R_c^{-n}(0)
$$
is the only $\beta$-summable $G_{J_{R_c}}$-orbit.
\end{enumerate}
\end{lemma}
\begin{proof} Assume that $\mathcal O$ is a $\beta$-summable
  $G_{J_{R_c}}$-orbit, $\beta > 0$, and let $m_{\mathcal O}$ be the
  corresponding purely atomic $\left(G_{J_{R_c}},L\right)$-conformal measure
  (\ref{mathcalo}). Since $\lim_{n \to \infty}  \left|\left(R_c^n\right)'(c)\right|^{\beta} = \infty$ we
  conclude that $c \notin \mathcal O$, and hence that $m_{\mathcal O}$
  is $\beta$-conformal by Lemma \ref{trans=conf4}. It follows therefore from item 3 of Corollary
  5.1 in \cite{GS2} that $\mathcal O =\bigcup_{n \geq 0}
  R_c^{-n}(0)$, i.e this is the only possible $\beta$-summable
  $G_{J_{R_c}}$-orbit when $\beta > 0$.

In case a) $0$ is pre-periodic, and it
  clearly can not be pre-periodic to a critical periodic
  orbit. Furthermore, it follows from Lemma 10 in \cite{GS1} that the
  Collet-Eckmann condition rules out the presence of a periodic
  neutral orbit in $J_{R_c}$.  It follows
  therefore from Lemma \ref{consistent} that $\mathcal O =
  G_{J_{R_c}}0$ is not consistent. There is therefore no
  $\beta$-summable $G_{J_{R_c}}$-orbit when $0$ is pre-periodic.

In case b) it follows from item 1 and 2 of
  Corollary 5.1 in \cite{GS2} that $\beta > HD(J_{R_c})$ since
  $m_{\mathcal O}$ is $\beta$-conformal and purely atomic, and then from
  item 3 of the same corollary that $\mathcal O =\bigcup_{n \geq 0}
  R_c^{-n}(0)$ actually \emph{is} $\beta$-summable when $\beta > HD(J_{R_c})$.
 
\end{proof}

\begin{thm}\label{Acon} Assume that $R_c$ satisfies the Collet-Eckmann
  condition and let
  $HD\left(J_{R_c}\right)$ be the Hausdorff dimension of the Julia set
  $J_{R_c}$. Consider the conformal
  action $\gamma$ on
    $C^*_r\left(G_{J_{R_c}}\right)$. Let $\beta > 0$.
\smallskip

a) When $0$ is pre-periodic for $R_c$ there
    is a $\beta$-KMS state for $\gamma$ if and only if
    $\beta = HD\left(J_{R_c}\right)$. It is given by 
\begin{equation}\label{M}
\omega(a) = \int_{J_{R_c}} P(a) \ dm,
\end{equation}
where $m$ is the $HD\left(J_{R_c}\right)$-conformal measure of
\cite{GS2}.

\smallskip

b) When $0$ is not pre-periodic there are 
\begin{enumerate}
\item[$\bullet$] no $\beta$-KMS states when $0 < \beta <
  HD\left(J_{R_c}\right)$,
\item[$\bullet$] a unique $\beta$-KMS state $\omega$, given by (\ref{M}), when $\beta =
  HD\left(J_{R_c}\right)$, and
\item[$\bullet$] two extremal $\beta$-KMS states when $\beta >
  HD\left(J_{R_c}\right)$. The corresponding measure on $J_{R_c}$ is
  purely atomic and supported on $\bigcup_{n \geq 0}
  R_c^{-n}(0)$. 
\end{enumerate}

\end{thm}
\begin{proof} Note that the isotropy group of $0$ in
  $G_{J_{R_c}}$ is $\mathbb Z_2$ by c) of Proposition
  \ref{isotropy}. Therefore a $\beta$-summable $G_{J_{R_c}}$-orbit will give rise to two
  extremal $\beta$-KMS states by Theorem \ref{decompGX}. Other than
  that all the
  statements follow by combining Lemma \ref{summ} and Theorem
  \ref{decompGX} with Corollary 5.1 of \cite{GS2}.
\end{proof}

Thanks to the results from \cite{GS2}, Theorem \ref{Acon} can be
extended to general rational maps $R$ satisfying
the Collet-Eckmann condition of \cite{GS1}, provided the critical
points in $J_R$ have the same valency. The only difference is
that the number of extreme $\beta$-KMS states increases, depending on the number and the valency of the critical
points in $J_R$, and on whether or not their orbits intersect. We have here restricted
attention to the quadratic case for concreteness and because it
simplifies the statement.

If we adopt the terminology of Bost
  and Connes from \cite{BoC}, the phase transition which occurs
  in Theorem \ref{Acon} at the inverse temperature $\beta
  =HD\left(J_{R_c}\right)$ when $0$ is not pre-periodic, is caused by
  spontaneous 
  symmetry breaking. To see this note that we can define an
  automorphism $\xi$ of $C^*_r\left(J_{R_c}\right)$ such that
$$
\xi(f)[x,k,\eta,y] = \frac{\eta'(x)}{\left|\eta'(x)\right|}
f[x,k,\eta,y] 
$$
when $f \in C_c\left(G_{J_{R_c}}\right)$. This automorphism commutes
with the conformal action and 
interchanges its two extremal $\beta$-KMS states when $\beta > HD\left(J_{R_c}\right)$ and $0$ is
not pre-periodic. Hence for all $\beta \geq HD\left(J_{R_c}\right)$
there is exactly one $\xi$-invariant $\beta$-KMS state for the conformal action.


\begin{remark}\label{example} Note that in case a) of Theorem
  \ref{Acon} there is a purely atomic $\beta$-conformal
measure for $\beta > HD\left(J_{R_c}\right)$ which is supported on the
backward orbit of the critical point $0$, cf. \cite{GS2}, and that
this measure is not $\left(G_{J_{R_c}},L\right)$-conformal because the
$G_{J_{R_c}}$-orbit of $0$ is not consistent. To show that there are
also rational maps with a $\left(G_{J_{R}},L\right)$-conformal
measure which is not conformal, consider the polynomial
$$
R(z) = z\left(1 +
    \frac{z}{2}\right)^2
$$ 
which has a critical point at $z = -2$. The critical
  value $0$ is a fixed point and $R^{-1}(0) =
  \left\{0,-2\right\}$. It follows that the point $0$ is its own $G_{J_R}$-orbit,
  i.e. $G_{J_R}0 = \{0\}$. Note that $0$ is a parabolic fixed point since
  $R'(0) = 1$. It follows that
  the Dirac measure at $0$ is $\left(G_{J_R},L\right)$-conformal for any
  exponent $\beta \in \mathbb R$. Since it is supported by a critical
  value it can not be $\beta$-conformal for any $\beta \neq 0$.   
\end{remark}


\section{Generalized gauge actions and their KMS states}

\subsection{Generalized gauge actions}
Consider again the setting of Section \ref{sec2} and let $f : X \to
\mathbb R$ be a continuous function. We can then define a homomorphism
$c_f : G_X \to \mathbb R$ such that 
\begin{equation}\label{c-f}
c_f[x,k,\eta,y] = \lim_{N \to \infty} \left(\sum_{i=0}^{N}
f\left(H^i(x)\right) -  \sum_{i=0}^{N-k}
f\left(H^i(y)\right)\right) .
\end{equation}
$c_f$ is continuous and we can consider the corresponding
one-parameter group $\alpha^f_t = \sigma^{c_f}_t$ of automorphisms on
$C^*_r\left(G_X\right)$. Following the general scheme laid out in
Section \ref{A} the KMS states of $\alpha^f$ can be determined from
the corresponding
non-atomic $\left(G_X,c_f\right)$-conformal measures and summable $G_X$-orbits. 

Let $\beta \in \mathbb R$. Following \cite{DU} we say that a finite Borel
measure $m$ on $X$ is \emph{$e^{\beta f}$-conformal} when
$$
m\left(H(A)\right) = \int_A e^{\beta f(x)} \ dm(x)
$$
for all Borel subsets $A$ of $X$ such that $H : A \to X$ is
injective. It is then easy to adopt the proofs of Lemma
\ref{n-mlemma4} and Lemma \ref{trans=conf4} to obtain the following.

\begin{lemma}\label{DUconf} Let $m$ be a finite Borel measure on $X$
  such that $m$ has no mass at the critical points or critical values
  of $H$ in $X$. Let
  $\beta \in \mathbb R \backslash \{0\}$. Then
  $m$ is $\left(G_X,c_f\right)$-conformal with exponent $\beta$ if and only if $m$ is
  $e^{\beta f}$-conformal.
\end{lemma}

In particular, the non-atomic $\left(G_X,c_f\right)$-conformal measures with exponent
$\beta$ coincide with the non-atomic $e^{\beta f}$-conformal Borel
measures.

Which $G_X$-orbits are
$\beta$-summable depends of course very much on the behaviour of
$f$. However, the case $f = 1$ is easy to handle as we do in the
following section.

\subsection{The gauge action}\label{gauge} 
 
The \emph{gauge action} on $C^*_r\left(G_X\right)$ is the generalised gauge
action $\alpha^1$ obtained from the homomorphism (\ref{c-f}) when $f$ is the constant function $1$. It is determined by the condition that
$$
\alpha^1_t(f)[x,k,\eta,y] = e^{ikt} f[x,k,\eta,y]
$$
when $f \in C_c\left(G_X\right)$.

Assume now that $S$ is the Riemann sphere and $H$ a rational map $R$
of degree at least $2$. Assume also that $X = J_R$, the Julia set of $R$. In this
case we can describe all $\beta$-KMS states for the gauge action on
$C^*_r\left(G_{J_R}\right)$ when $\beta \neq 0$.

\begin{lemma}\label{luybich} Let $m$ be a $e^{\beta}$-conformal Borel
  probability measure. Assume that $m$ is non-atomic. Then $\beta = \log d$
  and $m$ is the Lyubich measure of maximal entropy, cf. \cite{L}. 
\end{lemma}
 \begin{proof} Let $\mathcal C$ be the critical points of $R$. Every
   point of $J_R \backslash R(\mathcal C)$ is
   contained in an open set $U \subseteq J_R
   \backslash R(\mathcal C)$ such that $R^{-1}(U)$ is a disjoint union
   $R^{-1}(U) = \sqcup_{i=1}^d V_i$ where each $V_i$ is open and $R :
   V_i \to U$ is a homeomorphism. Let $B \subseteq U$ be a Borel
   subset and set $B_i = R^{-1}(B)\cap V_i$. Since $m(B_i) = e^{-\beta} m(B) =
   m(B_1)$ we conclude that $m(R^{-1}(B)) = dm(B_1) = de^{-\beta}
   m(B)$. Every Borel subset of $J_R \backslash
   R(\mathcal C)$ is the disjoint union of a countable collection of
   Borel sets each of which is a subset of an open set $U \subseteq
   J_R \backslash R(\mathcal C)$ as above. It follows
   therefore that 
\begin{equation}\label{7}
m\left(R^{-1}(B)\right) = de^{-\beta} m(B)
\end{equation}
for every Borel set $B \subseteq J_R$ because
$m\left(R\left(\mathcal C\right)\right) = 0$. Taking $B =
J_R$ yields the conclusion that $\beta = \log d$ as
   asserted. Once this is established it follows from (\ref{7}) that
   $m$ is $R$-invariant and then from the theorem in \cite{FLM} and
   \cite{L} that $m$ is the Lyubich measure.  
\end{proof}

\begin{lemma}\label{summa} Let $\mathcal O$ be a $\beta$-summable
  $G_{J_R}$-orbit, $\beta \neq 0$. It follows that
  there is a critical point $c$ for $R$ which not pre-periodic such
  that $\mathcal O = G_{J_R}c$.
\begin{proof} Let $x \in \mathcal O$. Since $J_R$ does not contain any
  critical periodic orbits the d1)-case in Proposition \ref{isotropy}
  does not arise. It follows therefore from Proposition \ref{isotropy} that the
  homomorphism $c_1[x,k,\eta,y]$ $ = k$ which defines the gauge action
  $\alpha^1$ does not annihilate $\left(G_{J_R}\right)^x_x$ when $x$ is pre-periodic. Thus $\mathcal
  O$ is is not consistent and hence not $\beta$-summable when $x$ is
  pre-periodic. Assume therefore that $x$ is not pre-periodic. To
  finish the proof it suffices to show that $G_{J_R}x$ is not summable unless $x$ is pre-critical. Assume
  therefore that no element of the forward orbit $\left\{ R^n(x) : \ n
    =0,1,2, \dots \right\}$ is critical. Since $R^n(x)  \in \mathcal
  O$ for all $n$ and $l_x\left(R^n(x)\right) = e^{n}$ we see
  that 
$\sum_{z \in \mathcal O} l_x(z)^{\beta} = \infty$,
unless $\beta < 0$. Consider therefore now the case $\beta < 0$. Since
$R^{n-1}(x)$ is not critical, and the degree of $R$ at least 2, there is an
element $x_n \in R^{-1}\left(R^n(x)\right)$ other than $R^{n-1}(x)$
when $n \geq 1$. Note that
$$
N = \left\{ n \geq 1: \ x_n \ \text{is not critical} \right\}
$$
must be infinite because $x$ is not pre-periodic. Assume to reach a contradiction that there is a critical point in the
backward orbit of $x_n$ for every $n \in N$. Since there are only
finitely many critical points there must then be some $ m
< n$ in $N$ such that the same critical point $c$ is contained in both the
backward orbit of $x_m$ and the backward orbit of $x_n$. Note
that $c$ can not be pre-periodic since $x$ is not. There is therefore a
unique $k \in \mathbb N$ such that $R^k(c) = R^n(x)$ and a unique $k'
\in \mathbb N$ such that $R^{k'}(c) = R^m(x)$. Then $R^{n-m+k'}(c) =
R^n(x) = R^k(c)$ and hence $n-m+k' = k$. Since $R^{n-m+k'-1}(c) =
R^{n-1}(x)$ while $R^{k-1}(c) = x_n \neq R^{n-1}(x)$, this is a contradiction. There
must therefore be an $n \in N$ such that the backward orbit of
$x_n$ does not contain any critical point. Choose $z_k \in
R^{-k}\left(x_n\right)$ and note that $z_k$ is an element of $\mathcal O$ for all
$k$. Since
$$
l_x(z_k)^{\beta} = e^{\beta(n-k -1)}
$$
we conclude that $\sum_{z \in \mathcal O} l_x(z)^{\beta} = \infty$,
also when $\beta < 0$.
\end{proof}
\end{lemma}

Let $\mathcal C_0$ be the (possibly empty) set of critical and
not pre-periodic points in $J_R$. Since the elements are not
pre-periodic the limit
$$
\VAL_{\infty}(c) = \lim_{n \to \infty} \val\left(R^n,c\right)
$$
exists for every $c \in \mathcal C_0$. We define an equivalence relation
$\sim$ on $\mathcal C_0$ such that $x \sim y$ if and only if
$\VAL_{\infty}(x) = \VAL_{\infty}(y)$ and $R^n(x) = R^m(y)$ for some
$n,m \in \mathbb N$. It follows from Lemma \ref{first} that two
elements $x,y \in \mathcal C_0$ are in the same $G_{J_R}$-orbit if and only if $x \sim y$. For each $\xi \in \mathcal C_0/{\sim}$ choose a
representative $c_{\xi} \in \mathcal C_0$ and set
$$
[\xi] = G_{J_R}c_{\xi} .
$$
It follows from Lemma \ref{summa} that these sets constitute the only
possible $\beta$-summable $G_{J_R}$-orbits, for any
$\beta \neq 0$.

\begin{lemma}\label{negpos} Let $d$ be the degree of $R$.
When $\beta > \log d$ the set $[\xi]$ is
$\beta$-summable for all $\xi \in \mathcal C_0/{\sim}$. When $\beta
\leq \log d$ the set $[\xi]$ is
$\beta$-summable  if and only if it is finite. 
\end{lemma}
\begin{proof} Note that
$$
[\xi] \ \subseteq \bigcup_{c \in \mathcal C_0, n
  \in \mathbb N} R^{-n}(c) 
$$
and $l_c(z) = e^{-n}$ when $z \in R^{-n}(c) \cap G_{J_R}c$. 
There is then a $K_{\beta} > 0$ such that
$$
K_{\beta}^{-1} e^{-\beta n} \leq l_{c_{\xi}}\left(z\right)^{\beta} \leq
K_{\beta}e^{-\beta n}
$$
for all $z \in [\xi] \cap R^{-n}\left(\mathcal C_0\right)$ and all $n$. Since an element
$c \in \mathcal C_0$ is not periodic, there is also a $C > 0$ such that
$$
C^{-1} d^n \leq \# [\xi] \cap R^{-n}(\mathcal C_0) \leq C d^n
$$
for all $n$, provided that $[\xi]$ is not finite.  
Hence
$$
K_{\beta}^{-1}C^{-1} e^{n(\log d - \beta)} \leq \sum_{z \in [\xi] \cap R^{-n}(\mathcal C_0)} l_{c_{\xi}}\left(z\right)^{\beta}
\leq CK_{\beta} e^{n(\log d - \beta)} 
$$
for all $n$ when $[\xi]$ is infinite. This proves the lemma.
\end{proof}

Let $\mathcal C_{00}$ be the set of critical and not pre-periodic points $c$
with the property that $G_{J_R}c$ is finite. For each $\xi \in \mathcal C_0/{\sim}$, set $\VAL_{\infty}(\xi) =
\VAL_{\infty}\left(c_{\xi}\right)$. Note that it follows from the
proof of Proposition \ref{isotropy} that $\VAL_{\infty}(\xi)$ is the order
of the cyclic isotropy groups $\left(G_{J_R}\right)^x_x, x \in [\xi]$. We can then summarise our findings
as follows.

\begin{thm}\label{Bcon} Let $R$ be a rational map of degree $d \geq 2$
  on the Riemann sphere with Julia set $J_R$ and consider the gauge
  action $\alpha^1$ on $C^*_r\left(G_{J_R}\right)$.

\smallskip

\begin{enumerate} 
\item[$\bullet$] When $0 \neq \beta < \log d$ there are exactly
$$
\sum_{\xi \in \mathcal C_{00}/\sim} \VAL_{\infty}[\xi]
$$
extremal $\beta$-KMS states for $\alpha^1$, all purely atomic. In
fact, the corresponding measures on $J_R$ have finite support. 

\item[$\bullet$] When $\beta = \log d$ there are exactly 
$$
1 \ + \sum_{\xi \in \mathcal C_{00}/\sim} \VAL_{\infty}[\xi]
$$
extremal $\beta$-KMS states for $\alpha^1$. One is non-atomic and the
associated measure is the Lyubich measure. The others are all purely
atomic and the associated measures have finite support.

\item[$\bullet$] When $\log d < \beta  < \infty$ there are exactly
$$
\sum_{\xi \in \mathcal C_{0}/\sim} \VAL_{\infty}[\xi]
$$
extremal $\beta$-KMS states for $\alpha^1$, all purely atomic.

\end{enumerate} 

\end{thm}

\begin{example} To give an example where there are KMS-states for
  which the associated measures have finite support we use the work of M. Rees. She shows in Theorem 2 of \cite{R} that for 'many' $\lambda
\in \mathbb C\backslash \{0\}$ the rational map
$$
R(z) = \lambda\left(1 - \frac{2}{z}\right)^2
$$
has a dense critical forward orbit. In particular, the Julia set $J_R$ is the
entire sphere. The critical points are $0$ and $2$, and
$R^{-1}(0) = \{2\}$. Hence the $G_{J_R}$-orbit of $0$ consists only of
the point $0$. Since $\VAL_{\infty}(0) =2$ there are for all $\beta
\neq 0$ exactly two extremal 
$\beta$-KMS states for the gauge action on $C^*_r\left(G_{J_R}\right)$
such that the associated measure is the Dirac measure at $0$. 
\end{example}

\end{document}